\documentclass [12pt]{amsart}
\usepackage[utf8]{inputenc}
\usepackage[margin=1in]{geometry}
\usepackage{amsmath,amssymb,amsthm,amsfonts}

\usepackage{mathtools}%

\usepackage[english]{babel}%
\usepackage{comment}%
\usepackage[unicode]{hyperref}%
\usepackage{bbm}%
\usepackage{bm}
\usepackage{mathrsfs}%
\usepackage[shortlabels]{enumitem}
\usepackage{cleveref}
\usepackage{parskip}

\usepackage{tikz,graphicx,color}
\usepackage{tikz-cd}%
\usepackage{tikz-3dplot}%
\usetikzlibrary{calc}%
\usetikzlibrary{arrows}%
\usetikzlibrary{shapes}%
\usetikzlibrary{patterns}%
\usetikzlibrary{positioning}%
\usetikzlibrary{decorations.markings, arrows.meta}
\usetikzlibrary{arrows.meta}
\usetikzlibrary{decorations.markings}
\usetikzlibrary{knots}

\usepackage{epstopdf}%

\usepackage[arrow]{xy}%
\usepackage{diagbox}%
\usepackage[normalem]{ulem}
\usepackage{subfig}%
\usepackage{arcs}%
\usepackage{xcolor}%
\usepackage{xspace}

\usepackage{enumitem}
\usepackage{comment}
\usepackage{letltxmacro}
\usepackage{thmtools,etoolbox}

\newtheorem{theorem}{Theorem}[section]

\newtheorem{lemma}[theorem]{Lemma}
\newtheorem{proposition}[theorem]{Proposition}
\newtheorem{corollary}[theorem]{Corollary}
\theoremstyle{definition}
\newtheorem{remark}[theorem]{Remark}
\theoremstyle{definition}
\newtheorem{definition}[theorem]{Definition}

\theoremstyle{definition}

\theoremstyle{definition}
\newtheorem{example}[theorem]{Example}
\theoremstyle{definition}

\DeclareMathOperator{\Gr}{Gr}
\DeclareMathOperator{\Deo}{Deo}
\DeclareMathOperator{\DS}{DS}
\newcommand{\C}{\mathbb{C}}
\newcommand{\F}{\mathbb{F}}

\newcommand{\cR}{\mathcal R}

\newcommand{\cM}{\mathcal M}
\newcommand{\Pio}{\Pi^{\circ}}
\newcommand{\Ro}{\mathring{\mathcal R}}
\newcommand{\Rfin}{\mathring R}
\newcommand{\Bkn}{\mathbf B_{k,n}}

\tikzset{
    node/.style = {circle, draw, fill=white, inner sep=0pt, minimum size=6pt, outer sep = 2pt},
    arrow/.style = {-{Stealth[length=5pt, width=4pt]}}
    amidarrow/.style  args={#1,#2}={postaction={decorate, decoration={markings, mark=at position #2 with {\arrow[scale=1.2]{Stealth[#1]}}}}}
}
\tikzset{square/.style={regular polygon, regular polygon sides=4, draw, inner sep=0,outer sep = 2pt,minimum size=8pt}}

\title{Toric Richardson Varieties and Slice Links}
\author{Thomas C. Martinez, Matthew J. Tyler }
\date{\today}


\begin{document}

\begin{abstract}
For type $A$ braid varieties, including Richardson varieties and positroid varieties, we prove that the standard torus has a dense orbit exactly when the associated link is smoothly slice. More generally, we relate the smooth slice genus to the codimension of a generic standard-torus orbit. For knots, this codimension equals twice the slice genus. We also characterize algebraic tori among open affine Richardson varieties and positroid patches by $2$-crown avoidance in their Bruhat intervals. Finally, positroid links have equal Seifert and slice genera, so a positroid link is smoothly slice exactly when it is an unlink.
\end{abstract}

\maketitle

\section{Introduction}\label{sec:intro}

We relate standard torus orbits on braid Richardson varieties to the topology of their associated links. Our first result characterizes dense orbits by smooth sliceness.

\subsection{Standard torus orbits and slice links}

Let $\beta=(s_{i_1},\ldots,s_{i_m})$ be a positive word on $n$ strands, with Demazure product $\delta(\beta)$. We also write $\beta$ for the corresponding braid and its image $s_{i_1}\cdots s_{i_m}\in S_n$. For $u\leq\delta(\beta)$, the braid Richardson variety $\Ro_{u,\beta}$ is nonempty and carries the standard action of $T_n=(\C^*)^n/\C^*$. Its associated link is
\[
 L_{u,\beta}:=\widehat{\beta\,\beta(u)^{-1}}\subset S^3,
\]
where $\beta(u)$ is the positive braid lift of $u$ \cite[Section~10.2]{GLSBS}.

The \emph{solid multigraph} $T_{u,\beta}$ has one vertex for each strand and one edge for each omitted position of the positive distinguished $u$-subword. The edge joins the two strand labels in the selected-subword wiring, and parallel edges are retained; see Section~\ref{sec:braid} and Definition~\ref{def:solid-graph}.

Write $c(\pi)$ for the number of cycles of a permutation $\pi$, including fixed points. A link is \emph{smoothly slice} if it bounds disjoint smooth disks in $B^4$.

\begin{theorem}\label{thm:braid-standard}
    Let $u\leq \delta(\beta)$. The following are equivalent.
    \begin{enumerate}[label=\textup{(\roman*)},nosep]
        \item $\Ro_{u,\beta}$ has a dense $T_n$-orbit.
        \item $\Ro_{u,\beta}$ is a single $T_n$-orbit.
        \item $T_{u,\beta}$ is a forest.
        \item $m-\ell(u) = n-c(\beta u^{-1})$.
        \item $L_{u,\beta}$ is smoothly slice.
    \end{enumerate}
\end{theorem}

The cycles of $\beta u^{-1}$ correspond to the components of $L_{u,\beta}$, so condition \textup{(iv)} compares the dimension of the variety with the number of link components. We call a variety with a dense standard-torus orbit \emph{standard toric}; the action need not be effective.

For finite Richardson varieties, the standard diagonal action is well understood. Can--Saha \cite[Theorem~4.7]{CanSaha} characterize standard toricity across finite types by linear independence of the roots in the positive Deodhar parametrization. In type $A$, Tsukerman--Williams \cite[Theorem~4.6 and Proposition~4.12]{tsukerman2015bruhat} compute the generic orbit dimension from a maximal-chain graph and obtain a forest criterion; see also Lee--Masuda--Park \cite{lee2021toric}.

Theorem~\ref{thm:braid-standard} extends the graph criterion to arbitrary positive type $A$ words and identifies standard toricity with sliceness of the associated link.

The proof uses the omitted positions of the positive distinguished subword in two ways. They give the weights of the standard action on the positive Deodhar torus and a quasipositive factorization of the link braid with $d=m-\ell(u)$ bands. The weights form an incidence matrix of $T_{u,\beta}$. If the graph is a forest, reflection length rules out every other Deodhar stratum. The dense orbit then fills the variety.

The corresponding \emph{ribbon surface} $S_{u,\beta}\subset B^4$ is the quasipositive disk-and-band surface associated with \eqref{eq:braid-quasipositive}. It has $n$ disks and $d=m-\ell(u)$ bands, with boundary $L_{u,\beta}$. In the 3D plabic-graph picture of \cite[Section~10.2]{GLSBS}, we resolve the ribbon intersections by pushing the interiors into the fourth coordinate while keeping the boundary fixed. The 3D plabic graph is a spine of the surface. Contracting each wiring strand to a point gives $T_{u,\beta}$. These strands are disjoint trees, so the contraction is a homotopy equivalence and $b_1(S_{u,\beta})=b_1(T_{u,\beta})$. Figure~\ref{fig:3D-graph-T-and-Link} shows the graph and the associated link.

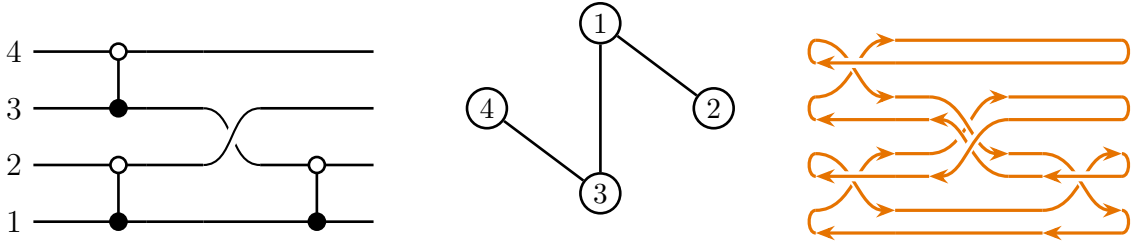
\begin{figure}
    \centering
    \begin{tikzpicture}[vertex/.style={circle,draw,fill=white,inner sep=0pt,minimum size=15pt,font=\small},overcross/.style={preaction={draw, white, line width=3.8pt, -}},link/.style={orange!90!black, very thick, -{Stealth[length=2.5mm]}},or/.style={orange!90!black, very thick},scale=1.5, line width=1pt]
     \foreach \y/\label in {0/1, 0.5/2, 1.0/3, 1.5/4} {
        \draw  (0,\y) -- (1,\y);
        \draw (1,\y) -- (1.5,\y);
        \node[left] at (0,\y) {$\label$};
        \node[right,black] at (3,\y) {
        };}

        \draw[thick] (1.5,1.0) to [out=0,in=180] (2.0,0.5);
   \draw[color=white, line width=5] (1.5,0.5) to [out=0,in=180] (2.0,1.0);
   \draw[thick] (1.5,0.5) to [out=0,in=180] (2.0,1.0);

     \draw (1.5,0) -- (3,0);
       \draw (2.0,0.5) -- (3,0.5);
      \draw  (2.0,1) -- (3,1);
       \draw (1.5,1.5) -- (3,1.5);

    \filldraw[fill=black] (0.75, 0.) circle (2pt); 
       \draw (0.75, 0.5) -- (0.75, 0.);
    \filldraw[fill=white] (0.75, 0.5) circle (2pt);
    
    \filldraw[fill=black] (0.75, 1.) circle (2pt); 
       \draw (0.75, 1.5) -- (0.75, 1.);
    \filldraw[fill=white] (0.75, 1.5) circle (2pt);
    
    \filldraw[fill=black] (2.5, 0.) circle (2pt); 
       \draw (2.5, 0.5) -- (2.5, 0.);
    \filldraw[fill=white] (2.5, 0.5) circle (2pt);

    \node[vertex] (1) at (5,1.75) {$1$};
    \node[vertex] (2) at (6,1) {$2$};
    \node[vertex] (3) at (5,0.25) {$3$};
    \node[vertex] (4) at (4,1) {$4$};

    \draw (1) to (2);
    \draw (1) to (3);
    \draw (3) to (4);

    \def\off{6.5}
    \draw[link,overcross] (\off+0.4,0.1) to [out=0,in=180] (\off+1.1,0.6);
    \draw[link,overcross] (\off+0.4,0.6) to [out=0,in=180] (\off+1.1,0.1);
    
    \draw[link,overcross](\off+1.1,0.4) to (\off+0.4,0.4);
    \draw[link,overcross] (\off+1.1,-0.1) to (\off+0.4,-0.1);

    \draw[link,overcross] (\off+0.4,1.1) to [out=0,in=180] (\off+1.1,1.6);
    \draw[link,overcross] (\off+0.4,1.6) to [out=0,in=180] (\off+1.1,1.1);
    
    \draw[link,overcross](\off+1.1,1.4) to (\off+0.4,1.4);
    \draw[link,overcross] (\off+1.1,0.9) to (\off+0.4,0.9);

    \draw[link,overcross] (\off+1.4,0.6) to [out=0,in=180] (\off+2.1,1.1);
    \draw[link,overcross] (\off+1.4,1.1) to [out=0,in=180] (\off+2.1,0.6);

    \draw[link,overcross] (\off+2.1,0.4) to [out=180,in=0] (\off+1.4,0.9);
    \draw[link,overcross] (\off+2.1,0.9) to [out=180,in=0] (\off+1.4,0.4);

    \draw[link,overcross] (\off+2.4,0.1) to [out=0,in=180] (\off+3.1,0.6);
    \draw[link,overcross] (\off+2.4,0.6) to [out=0,in=180] (\off+3.1,0.1);
    
    \draw[link,overcross](\off+3.1,0.4) to (\off+2.4,0.4);
    \draw[link,overcross] (\off+3.1,-0.1) to (\off+2.4,-0.1);

    \draw[or](\off+2.4,-0.1) to (\off+1.1,-0.1);
    \draw[or](\off+2.4,0.1) to (\off+1.1,0.1);
    \draw[or](\off+1.4,0.4) to (\off+1.1,0.4);
    \draw[or](\off+1.4,0.6) to (\off+1.1,0.6);
    \draw[or](\off+2.4,0.4) to (\off+2.1,0.4);
    \draw[or](\off+2.4,0.6) to (\off+2.1,0.6);

    \draw[or](\off+1.4,1-0.1) to (\off+1.1,0.9);
    \draw[or](\off+1.4,1.1) to (\off+1.1,1.1);
    
    \draw[or](\off+3.1,1-0.1) to (\off+2.1,0.9);
    \draw[or](\off+3.1,1.1) to (\off+2.1,1.1);
    
    \draw[or](\off+3.1,1.5-0.1) to (\off+1.1,1.4);
    \draw[or](\off+3.1,1.6) to (\off+1.1,1.6);

    \draw[or](\off+0.4,0.1) to [out=180, in=180] (\off+0.4,-0.1);
    \draw[or](\off+0.4,0.6) to [out=180, in=180] (\off+0.4,0.4);
    \draw[or](\off+0.4,1.1) to [out=180, in=180] (\off+0.4,0.9);
    \draw[or](\off+0.4,1.6) to [out=180, in=180] (\off+0.4,1.4);
    
    \draw[or](\off+3.1,0.1) to [out=0, in=0] (\off+3.1,-0.1);
    \draw[or](\off+3.1,0.6) to [out=0, in=0] (\off+3.1,0.4);
    \draw[or](\off+3.1,1.1) to [out=0, in=0] (\off+3.1,0.9);
    \draw[or](\off+3.1,1.6) to [out=0, in=0] (\off+3.1,1.4);

    \end{tikzpicture}
    \caption{The 3D plabic graph,  $T_{u,\beta}$, and the braid closure for $u=s_2$, $\beta=s_1s_3s_2s_1$.}
    \label{fig:3D-graph-T-and-Link}
\end{figure}

For an oriented link $L$ with $c$ components, let $\chi_4(L)$ be the maximum Euler characteristic of a smooth, compact, oriented surface in $B^4$ with boundary $L$, and let $\chi_3(L)$ be the analogous maximum for Seifert surfaces in $S^3$. Surfaces may be disconnected but have no closed components. We use the normalization
\[
 g_i(L)=\frac{c-\chi_i(L)}2,\qquad i\in\{3,4\}.
\]
For knots these are the usual genera. For links, $g_4(L)=0$ means that $L$ bounds disjoint smooth disks, and $g_3(L)=0$ means that $L$ is an unlink. These are not connected-surface genera: the positive Hopf link has $g_3=g_4=1$, although it bounds an annulus.

The graph also measures how far a generic orbit is from being dense. Let $\kappa=\#\pi_0(T_{u,\beta})$ be the number of graph components, let $c=c(\beta u^{-1})$ be the number of link components, and set
\[
 C_{u,\beta}:=\dim\Ro_{u,\beta}-\dim(T_n\cdot x)
\]
for a general point $x\in\Ro_{u,\beta}$.

\begin{theorem}\label{thm:genus-codimension}
Let $\beta$ be a positive word of length $m$ on $n$ strands, let
$u\leq\delta(\beta)$, and put $d=m-\ell(u)$. Then
\begin{equation}\label{eq:genus-codimension}
 C_{u,\beta}=d-n+\kappa=b_1(T_{u,\beta}),
 \qquad
 2g_4(L_{u,\beta})=C_{u,\beta}+c-\kappa,
\end{equation}
where $b_1(T_{u,\beta})$ is the first Betti number of the graph. In particular,
\[
 g_4(L_{u,\beta})\leq C_{u,\beta}
 \leq 2g_4(L_{u,\beta}).
\]
If $L_{u,\beta}$ is a knot, then
$C_{u,\beta}=2g_4(L_{u,\beta})=2g(S_{u,\beta})$.
\end{theorem}

For knots, the graph is connected and the ribbon surface $S_{u,\beta}$ is a minimal slice surface. For links, the correction term $c-\kappa$ is the difference between the numbers of link components and graph components. In both cases, the inequalities show that orbit codimension is zero exactly when the link is smoothly slice. The topological input is Rudolph's Euler-characteristic formula for quasipositive links \cite[Section~3]{Rudolph1993}. Theorem~\ref{thm:genus-codimension} expresses this formula in terms of the standard torus action.

For a reduced word, these results apply to ordinary Richardson varieties and their links; see Corollary~\ref{cor:finite-slice}. When the upper endpoint is Grassmannian, the open Richardson variety projects isomorphically to an open positroid variety $\Pio_f$, and its link is the positroid link $\Lambda_f$.

Write $\mathbf B_{k,n}$ for the $(k,n)$-bounded affine permutations, which index the open positroid varieties in $\Gr(k,n)$. For $f\in\mathbf B_{k,n}$ and $I\in\binom{[n]}k$, the corresponding Pl\"ucker patch is
\[
 \Pio_{f,I}:=\Pio_f\cap\{\Delta_I\neq0\}.
\]
Positroid links admit positive grid diagrams. This positivity implies that their Seifert and slice genera agree, so the slice criterion becomes an unlink criterion.

\begin{theorem}
\label{thm:intro-positroid-link} For every $f\in\mathbf B_{k,n}$, the positroid link $\Lambda_f$
satisfies
\[
g_3(\Lambda_f)=g_4(\Lambda_f).
\]
Consequently, the following are equivalent.
\begin{enumerate}[label=\textup{(\roman*)}]
 \item $\Pio_f$ is a single standard-torus orbit.
 \item Some, equivalently every, nonempty patch $\Pio_{f,I}$ is a single standard-torus orbit.
 \item $\Lambda_f$ is smoothly slice.
 \item $\Lambda_f$ is an unlink.
\end{enumerate}
\end{theorem}

We prove Theorem~\ref{thm:intro-positroid-link} in Section~\ref{sec:finite}. The equivalent criterion that a reduced plabic graph is a forest is due to \L{}ukowski--Parisi--Williams \cite[Propositions~3.15--3.16]{lukowski2023positive}; see Proposition~\ref{prop:plabic-forest}. This plabic graph differs from the solid multigraph used in the braid criterion.

\subsection{Intrinsic tori and affine patches}

The choice of torus action matters. A variety can be isomorphic to an algebraic torus without having a dense orbit under the standard action. We call a variety isomorphic to $(\C^*)^d$ an \emph{intrinsic torus}. An intrinsic torus that is not standard toric is \emph{unexpected}. This broadens the terminology of Gorsky--Kim--Sherman-Bennett \cite{unexpected}, who use the term for tori of dimension greater than $n-1$.

For $v\leq w$ in $S_n$, write $\Rfin_{v,w}$ and $\overline R_{v,w}$ for the open and closed Richardson varieties. Gorsky--Kim--Sherman-Bennett \cite[Theorem~1.2]{unexpected} prove that $\Rfin_{v,w}$ is an algebraic torus exactly when $\overline R_{v,w}$ is toric for some torus. They also show that these conditions are equivalent to $[v,w]$ containing no $2$-crown, or equivalently to $[v,w]$ being a lattice. Here a $2$-crown is a rank-three Bruhat interval isomorphic to the Bruhat order on $S_3$.

Bossinger--Gorsky--Simental \cite[Theorems~1.4 and~1.8]{bossinger2026torus} characterize intrinsic tori among braid varieties using double-root-free words and extend the finite Richardson criterion to finite crystallographic types. The torus supplied by this intrinsic structure can be larger than the standard diagonal torus. The link criteria above concern the standard action.

Our intrinsic criterion applies to affine type $A$ and, through affine Richardson varieties, to individual positroid patches. Let $f\leq g$ be affine permutations in the same component of the extended affine symmetric group, and let $\Ro_{f,g}$ be the corresponding open Richardson variety in the affine flag variety. We reserve $R_{f,g}(q)$ for its Kazhdan--Lusztig $R$-polynomial.

\begin{theorem}\label{thm:main}
    Let $f\leq g$ be affine permutations. The following are equivalent.
    \begin{enumerate}[label=\textup{(\roman*)}]
        \item $\Ro_{f,g}\cong(\C^*)^{\ell(g)-\ell(f)}$ as an algebraic variety.
        \item The affine Bruhat interval $[f,g]$ contains no 2-crown.
        \item $R_{f,g}(q)=(q-1)^{\ell(g)-\ell(f)}$.
        \item The Deodhar decomposition of $\Ro_{f,g}$ has a single stratum.
        \item For some, hence every, reduced word of $g$, the positive distinguished subexpression of $f$ is its only distinguished subexpression.
    \end{enumerate}
\end{theorem}

The proof starts with the positive Deodhar stratum, which is a dense open torus. An algebraic torus cannot contain a proper open subset that is itself a torus (Lemma~\ref{lem:open-torus}). Thus the whole variety is a torus exactly when the Deodhar decomposition has only this stratum.

To connect this condition with crown avoidance, we show that every Deodhar decomposition with more than one stratum has a stratum of codimension one (Lemma~\ref{lem:codim-one}). The argument replaces the prefix before the last length-decreasing step by a positive distinguished subword. It works for arbitrary words in arbitrary Coxeter systems and replaces the cluster-structure argument in the finite-type proof of \cite[Lemma~3.1]{unexpected}.

For a positroid patch, the upper endpoint of the affine interval is a translation. Given $f\in\mathbf B_{k,n}$ and $I\in\binom{[n]}k$, define
\[
 t_I(i)=\begin{cases} i+n,&i\bmod n\in I,\\ i,&i\bmod n\notin I \end{cases},
\]
where residues are represented by $[n]$. The patch is nonempty exactly when $I$ is a basis of the positroid of $f$, equivalently when $f\leq t_I$. In that case Snider's isomorphism \cite{snider,thomas} gives
\[
 \Pio_{f,I}\cong\Ro_{f,{t_I}},\qquad
 \dim\Pio_{f,I}=\ell(t_I)-\ell(f)=k(n-k)-\ell(f).
\]
The periodic path $P_I$ and affine Deograms in the following statement are recalled in Section \ref{sec:affine-patches}.

\begin{corollary}\label{cor:patch-main}
Let $f\in\Bkn$ and $I\in\binom{[n]}k$ satisfy $f\leq t_I$. The following are equivalent. 
\begin{enumerate}[label=\textup{(\roman*)}]
    \item $\Pio_{f,I}\cong(\C^*)^{k(n-k)-\ell(f)}$ as an algebraic variety.
    \item The affine Bruhat interval $[f,t_I]$ contains no 2-crown.
    \item $R_{f,t_I}(q)=(q-1)^{k(n-k)-\ell(f)}$.
    \item The Deodhar decomposition of $\Pio_{f,I}$ has a single stratum.
    \item There is a unique $(f,P_I)$-affine Deogram.
    \end{enumerate}
\end{corollary}

The examples in Section~\ref{sec:example} show why the distinction between intrinsic and standard toricity matters. For the top open positroid variety in $\Gr(2,4)$, the $\Delta_{13}$-patch is an intrinsic torus and the $\Delta_{12}$-patch is not. Standard toricity, however, is the same for every nonempty patch. The unexpected Richardson torus of Gorsky--Kim--Sherman-Bennett \cite[Example~2.11]{unexpected} has a nonslice Solomon link, so intrinsic toricity alone does not imply sliceness.

We also use the genus formula to recover the slice genus of the Perko knot. Finally, a seven-dimensional standard Richardson torus has the nontrivial slice knot $12n_{582}$ as its link. This example shows why the equality $g_3=g_4$ and the unlink characterization require the positroid hypothesis. The standard-torus slice criterion and the genus--codimension formula apply to braid Richardson varieties more generally.

\subsection*{Conventions and organization}

All varieties are over $\C$ unless a finite field is specified. Permutations compose right to left, so $(uv)(i)=u(v(i))$ and right multiplication by $s_i$ interchanges positions $i$ and $i+1$. We use $\ell$ and $\leq$ for Coxeter length and Bruhat order. For extended affine permutations in the same component, write $f=\omega\bar f$ and $g=\omega\bar g$ with $\omega$ of length zero. Reduced words, distinguished subexpressions, and $R_{f,g}$ refer to the Coxeter pair $\bar f\leq\bar g$.

Section~\ref{sec:braid} develops the Deodhar and graph calculations and proves the standard-torus slice criterion and the genus--codimension formula. Section~\ref{sec:finite} specializes these results to finite Richardson varieties and proves the stronger statement for positroid links. Sections~\ref{sec:affine-patches}--\ref{sec:affine-proof} establish the intrinsic criterion for affine Richardson varieties and positroid patches. Section~\ref{sec:standard} brings the intrinsic and standard criteria together on patches. Section~\ref{sec:example} illustrates the distinction between intrinsic and standard toricity, computes the slice genus of the Perko knot, and concludes with a nontrivial slice Richardson knot.

\section{Braid Richardson varieties}\label{sec:braid}

We prove Theorems~\ref{thm:braid-standard} and~\ref{thm:genus-codimension} using the omitted positions of the positive distinguished subword. These positions determine both the weights of the standard action and a quasipositive factorization of the associated link braid. We begin with the Deodhar decomposition and the intrinsic-torus facts needed here and in the affine argument.

\subsection{Positive words and distinguished subwords}

Let $G=\operatorname{GL}_n(\C)$, let $B_+$ and $B_-$ be its standard opposite Borel subgroups, and let $\mathcal F=G/B_+$. For flags $F,F'\in\mathcal F$ we write $F\xrightarrow{w}F'$ when their relative position is $w$. Fix a positive word $\beta=(s_{i_1},\dots,s_{i_m})$, not necessarily reduced. Its Demazure product is computed from $\delta_0=e$ by
\[
 \delta_j=\max(\delta_{j-1},\delta_{j-1}s_{i_j}),
 \qquad \delta(\beta)=\delta_m,
\]
the maximum being taken in Bruhat order. Assume $u\leq\delta(\beta)$. The braid Richardson variety is
\[
 \Ro_{u,\beta}
 =\left\{(F_0,\dots,F_m)\in\mathcal F^{m+1}:  \begin{array}{l}
 F_0=B_+,\quad F_{j-1}\xrightarrow{s_{i_j}}F_j
                 \text{ for }1\leq j\leq m,\\  F_m\in B_-\dot uB_+/B_+
 \end{array}\right\}.
\]
This endpoint-cell convention agrees with the braid Richardson model of \cite[Section~2.7]{GLSBS}. If $\beta$ is a reduced word for $w$, the endpoint map identifies $\Ro_{u,\beta}$ with
the open Richardson variety $\Rfin_{u,w}$.

A subexpression (or subword) of $\beta$ is a tuple $A=(a_1,\dots,a_m)$ of chosen letters, with $a_j\in\{e,s_{i_j}\}$. We write
\[
 a_{(0)}=e,\qquad a_{(j)}=a_1\cdots a_j\quad(1\leq j\leq m)
\]
for its prefix products. It is a distinguished $u$-subword if $a_{(m)}=u$ and
\[
 a_{(j)}\leq a_{(j-1)}s_{i_j}\quad(1\leq j\leq m).
\]
Thus a length-decreasing step must be taken, whereas a length-increasing step may be taken or skipped. We call position $j$ \emph{selected} if $a_j=s_{i_j}$ and \emph{omitted} if $a_j=e$. Write $\DS(u,\beta)$ for the set of distinguished $u$-subwords and put
\[
 e(A)=\#\{j:a_j=e\},\qquad
 b(A)=\#\{j:\ell(a_{(j)})<\ell(a_{(j-1)})\}.
\]
Counting selected increasing, selected decreasing, and omitted positions gives
\begin{equation}\label{eq:braid-balance}
 e(A)+2b(A)=m-\ell(u)=d.
\end{equation}
The set $\DS(u,\beta)$ is nonempty exactly when $u\leq\delta(\beta)$. These definitions and the construction below apply to arbitrary words in any Coxeter system $(W,S)$; the geometric applications use finite and affine type $A$.

For $u\leq\delta(\beta)$, the \emph{positive distinguished subword} $A_+=(p_1,\dots,p_m)$ is obtained by first computing its prefix products $u_j$ by the backward recursion
\begin{equation}\label{eq:pds-backward}
 u_m=u,\qquad u_{j-1}=\min(u_j,u_js_{i_j})\quad(j=m,\dots,1),
\end{equation}
and then setting $p_j=u_{j-1}^{-1}u_j\in\{e,s_{i_j}\}$. To see that $u_0=e$, choose any $A=(a_1,\dots,a_m)\in\DS(u,\beta)$. The lifting property makes $\pi_s(x):=\min(x,xs)$ order-preserving, and each step of $A$ satisfies $\pi_{s_{i_j}}(a_{(j)})\leq a_{(j-1)}$. Downward induction therefore gives $u_j\leq a_{(j)}$, hence $u_0=e$ and $u_j=p_1\cdots p_j$. The backward rule produces distinguished steps with no length decrease and is forced for every such subword. Thus $A_+$ is the unique element of $\DS(u,\beta)$ with $b=0$.

\subsection{Two elementary lemmas}

We use the following rigidity observation from \cite{unexpected}. We recall its proof for completeness.

\begin{lemma}[{\cite[Lemma~2.1]{unexpected}}]\label{lem:open-torus}
Let $T$ be an algebraic torus and let $U\subseteq T$ be a nonempty open subset. If $U$ is an algebraic torus, then $U=T$.
\end{lemma}

\begin{proof}
Since $U$ is open in the irreducible variety $T$ we have $\dim U=\dim T=:d$. Choose coordinates $y_1,\dots,y_d$ on $T$ and $x_1,\dots,x_d$ on $U$. Each restriction $y_i|_U$ is a unit in $\C[x_1^{\pm1},\dots,x_d^{\pm1}]$, hence a nonzero scalar $c_i$ times a Laurent monomial. Composing the inclusion $U\hookrightarrow T$ with translation by $(c_1^{-1},\ldots,c_d^{-1})$ therefore gives a homomorphism of algebraic groups. Its image is an open subgroup of the connected group $T$, hence is all of $T$. Translating back gives $U=T$.
\end{proof}

We will use the following lemma as well.

\begin{lemma}\label{lem:stable-open}
Let an algebraic torus $T$ act on an irreducible variety $X$, and let $U\subseteq X$ be a nonempty $T$-stable open subset. $X$ has a dense $T$-orbit if and only if $U$ does.
\end{lemma}

\begin{proof}
A dense orbit in $X$ meets $U$ and hence lies in $U$ by stability. Since $X$ is irreducible and $U$ is nonempty and open, $U$ is dense in $X$. Hence an orbit dense in $U$ is also dense in $X$.
\end{proof}

\subsection{The Deodhar decomposition and intrinsic tori}

The braid Richardson variety $\Ro_{u,\beta}$ is smooth, affine, and irreducible of dimension $d$ \cite[Proposition~6.13]{GLSBS}. Its Deodhar decomposition for an arbitrary positive word is
\begin{equation}\label{eq:braid-deodhar}
 \Ro_{u,\beta}
 =\bigsqcup_{A\in\DS(u,\beta)}D_A,
 \qquad D_A\cong(\C^*)^{e(A)}\times\C^{b(A)},
\end{equation}
and identifies the point count over $\F_q$ with the generalized $R$-polynomial
\begin{equation}\label{eq:braid-R}
 R_{u,\beta}(q)
 :=\sum_{A\in\DS(u,\beta)}
       (q-1)^{e(A)}q^{b(A)}.
\end{equation}
These statements hold for arbitrary positive words \cite[Section~7]{galashin2024rational}. The stratum $D_+:=D_{A_+}$ is open and isomorphic to $(\C^*)^d$ \cite[Proposition~7.3]{GLSBS}; by \eqref{eq:braid-balance} every other stratum has dimension $d-b(A)<d$. When $\beta$ is a reduced word for $w$, \eqref{eq:braid-R} is the ordinary Kazhdan--Lusztig $R$-polynomial $R_{u,w}(q)$.

\begin{lemma}\label{lem:braid-intrinsic}
For $u\leq\delta(\beta)$, the following are equivalent:
\[
 \Ro_{u,\beta}\cong(\C^*)^d,\qquad
 \DS(u,\beta)=\{A_+\},\qquad
 R_{u,\beta}(q)=(q-1)^d.
\]
When these conditions hold, $\Ro_{u,\beta}=D_+$.
\end{lemma}

\begin{proof}
If $\Ro_{u,\beta}$ is a torus, its open torus $D_+$ equals
$\Ro_{u,\beta}$ by Lemma~\ref{lem:open-torus}. Since every
distinguished subword indexes a nonempty Deodhar stratum, this is
equivalent to $\DS(u,\beta)=\{A_+\}$; conversely, that equality
exhibits $\Ro_{u,\beta}=D_+\cong(\C^*)^d$.

Substituting $q=1+z$ in \eqref{eq:braid-R} gives
\[
 R_{u,\beta}(1+z)
 =z^d+\sum_{A\neq A_+}z^{e(A)}(1+z)^{b(A)}.
\]
Every additional summand is nonzero with nonnegative coefficients.
Thus $R_{u,\beta}(q)=(q-1)^d$ exactly when the additional sum is empty.
\end{proof}

By Lemma~\ref{lem:braid-intrinsic}, uniqueness of the distinguished subword is invariant under positive braid moves, although the individual subwords and their parametrizations can change. For reduced $\beta$, the polynomial criterion concerns the corresponding Bruhat interval. For nonreduced $\beta$, crown avoidance in $[u,\delta(\beta)]$ does not characterize intrinsic toricity; see Example~\ref{ex:braid-demazure}.

The next lemma is the key combinatorial step in the affine criterion in Section~\ref{sec:affine-proof}. Let $(W,S)$ be any Coxeter system, let $\beta$ be an arbitrary word in $S$, and fix $u\leq\delta(\beta)$. In the geometric applications, $b(A)$ is the codimension of $D_A$, so the lemma says that a decomposition with more than one stratum must have one of codimension one.

\begin{lemma}\label{lem:codim-one}
If $\DS(u,\beta)\neq\{A_+\}$, then some $A\in\DS(u,\beta)$ has $b(A)=1$.
\end{lemma}

\begin{proof}
Choose $A=(a_1,\ldots,a_m)\in\DS(u,\beta)$ with $A\neq A_+$; then $b(A)\geq1$, since $A_+$ is the only distinguished $u$-subword with $b=0$. Let $j$ be the largest position at which $A$ takes a length-decreasing step, so that $a_j=s_{i_j}$ and $\ell(a_{(j)})<\ell(a_{(j-1)})$, with no position after $j$ length-decreasing.

The truncation $(a_1,\ldots,a_{j-1})$ is a distinguished subword of $(s_{i_1},\ldots,s_{i_{j-1}})$ with product $a_{(j-1)}$. Apply \eqref{eq:pds-backward} to this prefix word with endpoint $a_{(j-1)}$ to obtain its positive distinguished subword $(a_1',\ldots,a_{j-1}')$. It has the same product and no length-decreasing step. Set
\[
 A'=(a_1',\ldots,a_{j-1}',a_j,a_{j+1},\ldots,a_m).
\]
The prefix products of $A'$ and $A$ agree from position $j-1$ onward, so every retained suffix step remains distinguished and the full product is still $u$. The replacement prefix is positive distinguished, hence $A'\in\DS(u,\beta)$. Its only length-decreasing position is $j$, so $b(A')=1$.
\end{proof}

The following is the coefficient criterion from the proof of Gorsky--Kim--Sherman-Bennett \cite[Lemma~3.1]{unexpected}, with arbitrary positive words allowed in place of reduced words. Assume $d\geq2$ and let $m_{u,\beta}$ be the number of mutable variables in a seed of the cluster structure on $\Ro_{u,\beta}$.

\begin{corollary}\label{cor:coefficient}
The following are equivalent:
\[
 \Ro_{u,\beta}\text{ is a torus},\qquad
 m_{u,\beta}=0,\qquad
 [q^{d-1}]R_{u,\beta}(q)=-d .
\]
\end{corollary}

\begin{proof}
In the Deodhar seed, the mutable variables index the closures of the codimension-one Deodhar strata \cite[Sections~7.2--7.3]{GLSBS}, so $m_{u,\beta}$ also counts these strata. By \eqref{eq:braid-balance} a stratum $D_A$ has codimension $b(A)$, and those with $b(A)=1$ have $e(A)=d-2$; strata with $b(A)\geq2$ contribute terms of degree $e(A)+b(A)=d-b(A)\leq d-2$. Thus
\[
 R_{u,\beta}(q)
 =(q-1)^d+m_{u,\beta}\,q(q-1)^{d-2}
   +\bigl(\text{terms of degree at most $d-2$}\bigr).
\]
Extracting the coefficient of $q^{d-1}$ gives $-d+m_{u,\beta}$. If $\Ro_{u,\beta}$ is a torus then $\DS(u,\beta)=\{A_+\}$ by Lemma~\ref{lem:braid-intrinsic} and $m_{u,\beta}=0$; conversely $m_{u,\beta}=0$ forces $\DS(u,\beta)=\{A_+\}$ by Lemma \ref{lem:codim-one}.
\end{proof}

Let $N_{u,\beta}=|\DS(u,\beta)|$. For the empty word set $N_{e,\varnothing}=1$ and $N_{u,\varnothing}=0$ for $u\neq e$. For a positive word $\gamma$ and a simple reflection $s$, the counts satisfy
\begin{equation}\label{eq:braid-count-recursion}
 N_{u,\gamma s}=
 \begin{cases}
 N_{us,\gamma},&us<u,\\
 N_{us,\gamma}+N_{u,\gamma},&us>u.
 \end{cases}
\end{equation}

\begin{proposition}\label{prop:braid-algorithm}
The variety $\Ro_{u,\beta}$ is a torus if and only if $N_{u,\beta}=1$.
\end{proposition}

\begin{proof}
If $us<u$, a distinguished subword with product $u$ cannot omit its last letter, so that letter is $s$ and the preceding prefix product is $us$. If $us>u$, its last letter is either $e$, with preceding prefix product $u$, or $s$, with preceding prefix product $us$ and a decreasing last step. Deleting that letter gives exactly the two cases of \eqref{eq:braid-count-recursion}; this is the unweighted Deodhar recursion of \cite[Section~4.2]{galashin2024rational}. The decision criterion is Lemma~\ref{lem:braid-intrinsic}.
\end{proof}

\begin{example}\label{ex:braid-demazure}
Take $S_2=\{e,s\}$ and $\beta=(s,s)$, so that $\delta(\beta)=s$ and $[e,\delta(\beta)]=[e,s]$ is crown-free. There are nevertheless two distinguished $e$-subwords: $\mathbf{v} = (e,e)$ and $\mathbf{v}' = (s,s)$. Their strata are $(\C^*)^2$ and $\C$, so
\[
 R_{e,(s,s)}(q)=(q-1)^2+q
\]
and $\Ro_{e,(s,s)}$ is not a torus. By contrast $\Ro_{e,(s)}\cong\C^*$, although the Demazure product is again $s$. Nonreducedness is not by itself an obstruction: for the same word and endpoint $u=s$ the only distinguished subword is $\mathbf{v}''=(e,s)$, so
\[
 R_{s,(s,s)}(q)=q-1,\qquad
 \Ro_{s,(s,s)}\cong\C^* .
\]
All three point counts are immediate from the flag model. For instance, identifying $\mathcal F$ with $\mathbb P^1$, the variety $\Ro_{e,(s,s)}$ consists of the pairs $(F_1,F_2)$ with $F_1\neq B_+$, $F_2\neq F_1$, and $F_2$ in the big cell, of which there are $q+(q-1)^2=q^2-q+1$ over $\F_q$.
\end{example}

\subsection{Solid crossings and the standard action}

To determine the dimensions of standard-torus orbits, we compute the weights of the action on the positive Deodhar torus. These weights are indexed by the omitted positions of the positive distinguished subword and can be recorded as edges of a multigraph.

Write the positive distinguished subword as $A_+=(p_1,\ldots,p_m)$, with $p_j\in\{e,s_{i_j}\}$, and put
\[
 u_0=e,\qquad u_j=p_1\cdots p_j,
\qquad
 J^\circ=J^\circ(u,\beta)=\{j:p_j=e\},
\]
so that $u_m=u$ and $|J^\circ|=d$. Both selected and omitted positions satisfy $u_{j-1}s_{i_j}>u_{j-1}$; in particular the selected letters form a reduced expression for $u$, even though $\beta$ need not be reduced.

\begin{definition}\label{def:solid-graph}
Let $\epsilon_1,\ldots,\epsilon_n$ be the standard basis of $\mathbb Z^n$ and put $\alpha_i=\epsilon_i-\epsilon_{i+1}$. At a position $j\in J^\circ$ set
\[
 \rho_j=u_{j-1}(\alpha_{i_j})
       =\epsilon_{a_j}-\epsilon_{b_j},
 \qquad
 \tau_j=u_{j-1}s_{i_j}u_{j-1}^{-1}=(a_j\ b_j).
\]
The multigraph $T_{u,\beta}$ has vertex set $[n]$ and one edge $e_j=\{a_j,b_j\}$ for each $j\in J^\circ$. Equivalently, trace the wiring of the selected subword and insert a bridge between the two strands at each omitted position; see Figure~\ref{fig:3D-graph-T-and-Link}. Edges are indexed by solid positions, so parallel edges are not identified; two parallel edges already constitute a cycle. For a reduced word $\mathbf w$ we also write $T_{v,\mathbf w}$.
\end{definition}

\begin{remark}
The omitted PDS positions are the solid crossings of \cite[Sections~3.1--3.2]{GLSBS}. The strand labels in Definition \ref{def:solid-graph} are determined by the selected prefix $u_{j-1}$, not by the product in $S_n$ of all preceding letters of $\beta$.
\end{remark}

Let $k$ be the number of connected components of $T_{u,\beta}$. We use inverse Deodhar parameters $x_j\in\C^*$, indexed by $J^\circ$. With $y_i(z)=1+zE_{i+1,i}$ and permutation-matrix representatives $\dot s_i$, the corresponding gallery is
\[
 F_r=g_1\cdots g_rB_+,\qquad
 g_j=
 \begin{cases}
  y_{i_j}(x_j^{-1}),&j\in J^\circ,\\
  \dot s_{i_j},&j\notin J^\circ .
 \end{cases}
\]
This is the positive case of the Marsh--Rietsch parametrization \cite[Proposition~5.2]{MR04}, applied successively to the gallery; the Deodhar construction for arbitrary words is recalled in \cite[Theorem~7.4]{galashin2024rational}.

\begin{proposition}\label{prop:braid-weights}
The standard action on $D_+$ has weights $\{\rho_j:j\in J^\circ\}$ in the coordinates $x_j$, and every orbit in $D_+$ has dimension $n-k$.
\end{proposition}

\begin{proof}
For a diagonal matrix $h$,
\[
 h\,y_i(x^{-1})=y_i\bigl((\alpha_i(h)x)^{-1}\bigr)h,
 \qquad
 h\,\dot s_i=\dot s_i\,s_i(h).
\]
Thus moving $t=\operatorname{diag}(t_1,\ldots,t_n)$ through the gallery factors changes the diagonal factor only at selected positions. Immediately before position $j$ it is $u_{j-1}^{-1}(t)$, so at an omitted position
\[
 x_j(t\cdot F_\bullet)
 =\alpha_{i_j}\bigl(u_{j-1}^{-1}(t)\bigr)x_j(F_\bullet)
 =\frac{t_{a_j}}{t_{b_j}}x_j(F_\bullet)
 =\rho_j(t)x_j(F_\bullet).
\]
The remaining right diagonal factors lie in $B_+$ and do not change the flags. The coordinate formula is unchanged when $t$ is multiplied by a scalar, so it defines the standard $T_n$-action.

Since every $x_j$ is nonzero, a diagonal matrix fixes a point of $D_+$ exactly when $t_{a_j}=t_{b_j}$ for every edge of $T_{u,\beta}$. Its entries are therefore constant on each connected component. The stabilizer in $T_n$ has dimension $k-1$, and every orbit in $D_+$ has dimension $(n-1)-(k-1)=n-k$.
\end{proof}

Since $D_+$ is open and dense of dimension $d$, Lemma \ref{lem:stable-open} shows that a dense standard orbit exists exactly when $n-k=d$. For a multigraph with $n$ vertices and $d$ edges this equality is equivalent to being a forest. Thus $\Ro_{u,\beta}$ has a dense standard orbit if and only if $T_{u,\beta}$ is a forest.

\subsection{Reflection length and the absence of boundary strata}

The weight calculation determines when a dense standard orbit exists. To prove that such an orbit fills the variety, we must also exclude the remaining Deodhar strata. A factorization of $\beta u^{-1}$ into transpositions gives a bound on their codimensions and forces the decomposition to have a single stratum in the forest case.

Write $\ell_T(\pi)$ for the minimum number of transpositions whose product is $\pi\in S_n$. Then
\[
 \ell_T(\pi)=n-c(\pi),
\]
since every transposition changes the cycle count by one and a cycle of length $r$ is a product of $r-1$ transpositions. 

For the next lemma, let $A=(a_1,\ldots,a_m)$ be any subword of $\beta$ with product $a_{(m)}=u$, and let $J(A)=\{j:a_j=e\}$ be its omitted positions. The subword need not be distinguished. An arrow over a product means that its factors are ordered by increasing position. The following identity is in $S_n$.

\begin{lemma}\label{lem:omitted-factorization}
\begin{equation}\label{eq:hollow-reflection-factorization}
 \beta u^{-1}
 =\prod_{j\in J(A)}^{\longrightarrow}
   a_{(j-1)}s_{i_j}a_{(j-1)}^{-1}.
\end{equation}
\end{lemma}

\begin{proof}
Put $W_j=s_{i_1}\cdots s_{i_j}$ and $C_j=W_ja_{(j)}^{-1}$. If $a_j=s_{i_j}$ then $C_j=C_{j-1}$; if $a_j=e$ then $C_j=C_{j-1}\bigl(a_{(j-1)}s_{i_j}a_{(j-1)}^{-1}\bigr)$. Starting from $C_0=e$ proves the factorization.
\end{proof}

Each factor is a transposition. For a distinguished subword $A$, \eqref{eq:braid-balance} therefore gives
\begin{equation}\label{eq:braid-reflection-bound}
 \ell_T\bigl(\beta u^{-1}\bigr)\leq e(A)=d-2b(A).
\end{equation}
For $A_+$, the factors are the edge transpositions of $T_{u,\beta}$. The criterion below also shows that the forest property is invariant under positive braid moves, since $d$ and the permutation represented by $\beta$ are unchanged.

\begin{lemma}\label{lem:hollow-forest}
The multigraph $T_{u,\beta}$ is a forest if and only if $d=n-c\bigl(\beta u^{-1}\bigr)$.
\end{lemma}

\begin{proof}
Apply \eqref{eq:hollow-reflection-factorization} to $A_+$; its factors are exactly the edge transpositions $\tau_j$. If $T_{u,\beta}$ is a forest, each successive edge joins different components of the graph formed by the earlier edges. The preceding partial product preserves those components, so its cycles lie inside them and the new transposition merges two distinct cycles. Starting from $n$ cycles leaves $n-d$ cycles.

Conversely, the full product preserves each component of $T_{u,\beta}$, so $c(\beta u^{-1})\geq k$, where $k$ is the number of components. A multigraph on $n$ vertices with $d$ edges that is not a forest has $d>n-k$, hence $k>n-d$ and $c(\beta u^{-1})>n-d$, ruling out the displayed equality.
\end{proof}

\begin{proposition}\label{prop:braid-single-orbit}
If $d=\ell_T\bigl(\beta u^{-1}\bigr)$, then $\Ro_{u,\beta}$ is a single standard-torus orbit.
\end{proposition}

\begin{proof}
For any $A\in\DS(u,\beta)$, \eqref{eq:braid-reflection-bound} gives $d\leq d-2b(A)$, so $b(A)=0$ and hence $A=A_+$. The Deodhar decomposition therefore reduces to $\Ro_{u,\beta}=D_+$. By Lemma \ref{lem:hollow-forest} the multigraph $T_{u,\beta}$ is a forest, so its $d$ weights are linearly independent by Proposition \ref{prop:braid-weights}. The homomorphism from $T_n$ to the coordinate torus has image of dimension $d$ and is therefore surjective, so the action on $D_+$ is transitive.
\end{proof}

The stabilizer is
\[
 K=\{[t_1,\ldots,t_n]\in T_n:
       t_a=t_b\text{ whenever }a\text{ and }b
       \text{ lie in the same component of }T_{u,\beta}\},
\]
a subtorus of dimension $k-1$; in the forest case $\Ro_{u,\beta}\cong T_n/K\cong(\C^*)^d$ equivariantly. In particular the single-orbit property does not assert that the full $(n-1)$-dimensional torus acts effectively.

\subsection{Quasipositivity, sliceness, and slice genus}

The omitted-position factorization also lifts to the braid group. Its factors give a quasipositive surface bounded by $L_{u,\beta}$, and Rudolph's theorem determines its maximal Euler characteristic. This connects the orbit calculation with smooth sliceness and slice genus. Recall our normalization:
\begin{equation}\label{eq:link-genera}
 g_i(L)=\frac{c-\chi_i(L)}2,\qquad i\in\{3,4\}.
\end{equation}
A surface with $r$ components, $c$ boundary circles, and total genus $g$ has $\chi=2r-2g-c$, so its contribution to this normalization is $g+c-r$. Since every component has boundary, $r\leq c$; equality $\chi=c$ holds precisely for $c$ disjoint disks. This proves the zero-genus characterizations of sliceness and the unlink used above.

A braid is \emph{quasipositive} if it is a product of conjugates of positive Artin generators. Write $L=L_{u,\beta}$ and $c=c(\beta u^{-1})$ for its number of components.

\begin{lemma}\label{lem:braid-quasipositive}
The braid $\beta\,\beta(u)^{-1}$ is quasipositive with $d$ factors, and
\begin{equation}\label{eq:braid-chi}
 \chi_4(L)=n-d,\qquad g_4(L)=\frac{d-n+c}{2}.
\end{equation}
\end{lemma}

\begin{proof}
Let $B_j=\sigma_{i_1}\cdots\sigma_{i_j}$ and let $P_j$ be the braid formed by the selected letters of $A_+$ among the first $j$ positions. The selected subword is reduced, so $P_j=\beta(u_j)$ and in particular $P_m=\beta(u)$. If position $j$ is selected then $B_jP_j^{-1}=B_{j-1}P_{j-1}^{-1}$; if it is omitted then
\[
 B_jP_j^{-1}
 =B_{j-1}P_{j-1}^{-1}
   \bigl(P_{j-1}\sigma_{i_j}P_{j-1}^{-1}\bigr).
\]
Therefore
\begin{equation}\label{eq:braid-quasipositive}
 \beta\,\beta(u)^{-1}
 =\prod_{j\in J^\circ}^{\longrightarrow}
   \beta(u_{j-1})\sigma_{i_j}\beta(u_{j-1})^{-1},
\end{equation}
a quasipositive factorization with $d$ factors. Its braided surface in $B^4$ has $n$ disks and $d$ bands, hence Euler characteristic $n-d$, and the slice--Bennequin inequality gives the matching upper bound for any smooth bounding surface \cite[Section~3]{Rudolph1993}. Thus $\chi_4(L_{u,\beta})=n-d$. The closure has $c(\beta u^{-1})$ components, which gives the genus formula.
\end{proof}

\begin{proof}[Proof of Theorem \ref{thm:braid-standard}]
Proposition \ref{prop:braid-weights} and the discussion following it give \textup{(i)}$\Leftrightarrow$\textup{(iii)} and Lemma \ref{lem:hollow-forest} gives \textup{(iii)}$\Leftrightarrow$\textup{(iv)}. Proposition \ref{prop:braid-single-orbit} gives \textup{(iv)}$\Rightarrow$\textup{(ii)}, and \textup{(ii)}$\Rightarrow$\textup{(i)} is immediate. By \eqref{eq:braid-chi}, smooth sliceness is equivalent to $n-d=c(\beta u^{-1})$, which gives \textup{(iv)}$\Leftrightarrow$\textup{(v)}, with the genus given by \eqref{eq:braid-chi}.
\end{proof}

\begin{proof}[Proof of Theorem~\ref{thm:genus-codimension}]
Proposition~\ref{prop:braid-weights} gives the generic orbit dimension $n-\kappa$. Since $\dim\Ro_{u,\beta}=d$ and $T_{u,\beta}$ has $n$ vertices and $d$ edges,
\[
 C_{u,\beta}=d-n+\kappa=b_1(T_{u,\beta}).
\]
Lemma~\ref{lem:braid-quasipositive} therefore gives
\[
 2g_4(L_{u,\beta})=d-n+c=C_{u,\beta}+c-\kappa.
\]

We claim that $\kappa\leq c\leq\kappa+C_{u,\beta}$. The first inequality holds because the edge-transposition product \eqref{eq:hollow-reflection-factorization} preserves every connected component of $T_{u,\beta}$.

For the second inequality, insert the edges in their order in $\beta$. Exactly $n-\kappa$ edges join previously distinct graph components. Each corresponding transposition merges two permutation cycles, since the preceding product preserves every existing graph component. Each of the remaining $C_{u,\beta}$ transpositions can increase the cycle count by at most one. Starting with $n$ cycles gives
\[
 c\leq n-(n-\kappa)+C_{u,\beta}=\kappa+C_{u,\beta}.
\]
Consequently
\[
 C_{u,\beta}\leq 2g_4(L_{u,\beta})\leq 2C_{u,\beta}.
\]
If the link is a knot, then $c=1$ forces $\kappa=1$, giving $C_{u,\beta}=2g_4(L_{u,\beta})$. The ribbon surface is connected and has one boundary component. Since it consists of $n$ disks and $d$ bands, $\chi(S_{u,\beta})=n-d$, and therefore
\[
 2g(S_{u,\beta})  =1-\chi(S_{u,\beta})  =d-n+1  =2g_4(L_{u,\beta}).
\]
Thus $S_{u,\beta}$ is a minimal slice surface.
\end{proof}

Let $\overline X$ be an irreducible $T_n$-equivariant compactification of $X=\Ro_{u,\beta}$.

\begin{corollary}\label{cor:braid-compactification}
$\overline X$ has a dense standard-torus orbit if and only if the conditions of Theorem \ref{thm:braid-standard} hold.
\end{corollary}

\begin{proof}
Apply Lemma \ref{lem:stable-open} to the dense stable open subset $X\subseteq\overline X$.
\end{proof}

Here toricity refers to the dense-orbit property; normality of an arbitrary compactification is not asserted. One may take the closure in the Bott--Samelson gallery variety, whose boundary also allows successive flags to coincide.

For the next corollary, suppose $\DS(u,\beta)=\{A_+\}$.

\begin{corollary}\label{cor:braid-unexpected}
$\Ro_{u,\beta}$ is an unexpected torus if and only if $L_{u,\beta}$ is not smoothly slice.
\end{corollary}

\begin{proof}
The hypothesis gives an intrinsic torus by Lemma~\ref{lem:braid-intrinsic}, and Theorem \ref{thm:braid-standard} identifies the standard-toric cases with equality, equivalently with sliceness. Equation \eqref{eq:braid-reflection-bound} applied to $A_+$ gives $n-c(\beta u^{-1})\leq d$, so failure of equality is strict inequality.
\end{proof}

Equivalently, the unexpected cases are those with $\DS(u,\beta)=\{A_+\}$ and $d>n-c\bigl(\beta u^{-1}\bigr)$.

For a nonempty braid Richardson variety, compute $d=m-\ell(u)$ and the cycle count of the permutation $\beta u^{-1}$, and then test $d=n-c(\beta u^{-1})$.

\begin{proposition}\label{prop:braid-standard-algorithm}
This test decides standard toricity, the single-orbit property, and smooth sliceness simultaneously.
\end{proposition}

\begin{proof}
The standard test is Theorem \ref{thm:braid-standard}\,\textup{(iv)} and the intrinsic test is Proposition \ref{prop:braid-algorithm}. Standard toricity implies that the intrinsic test succeeds, by Proposition \ref{prop:braid-single-orbit}. Hence a positive standard test gives a standard torus, a negative standard test together with a positive intrinsic test gives an unexpected torus, and a negative intrinsic test gives a variety that is not a torus.
\end{proof}

Combined with Proposition \ref{prop:braid-algorithm}, the test distinguishes standard tori, unexpected tori, and varieties that are not tori.

\begin{example}\label{ex:braid-standard-nonreduced}
For $\beta=(s,s)$ in $S_2$ and $u=s$, Example \ref{ex:braid-demazure} gives $\Ro_{s,\beta}\cong\C^*$. Here the permutation represented by $\beta$ is $e$, and
\[
 d=1=2-c(s),\qquad
 L_{s,\beta}=\widehat{\sigma_1^2\sigma_1^{-1}}
           =\widehat{\sigma_1},
\]
so the standard action is transitive and the link is the unknot. For the same word with $u=e$, the solid graph has two parallel edges, the link is the positive Hopf link, and the variety is not an intrinsic torus.
\end{example}

\section{Finite Richardson varieties and positroid links}\label{sec:finite}

The standard-action criterion and genus formula specialize to ordinary Richardson varieties by taking the braid word to be reduced. Let $v\leq w$ in $S_n$, let $\mathbf w$ be a reduced word for $w$, and put $d=\ell(w)-\ell(v)$. The endpoint map is a $T_n$-equivariant isomorphism
\[
 \Ro_{v,\mathbf w}\xrightarrow{\ \sim\ }\Rfin_{v,w},
\]
and the associated braid Richardson link is the Richardson link
\[
 L_{v,\mathbf w}=\Lambda_{v,w}=\widehat{\beta(w)\beta(v)^{-1}}
\]
of Galashin--Lam \cite{galashin2024positroids}. We first compare the resulting slice criterion with the intrinsic classification, then use positivity to obtain the stronger conclusion for positroid links.

\subsection{Intrinsic versus standard toricity}

\begin{theorem}[Gorsky--Kim--Sherman-Bennett
{\cite[Theorem~1.2]{unexpected}}]\label{thm:GKSB}
For $v\leq w$ in $S_n$ the following are equivalent:
\begin{enumerate}[label=\textup{(\roman*)}]
 \item $\Rfin_{v,w}\cong(\C^*)^d$ as an algebraic variety;
 \item $\overline R_{v,w}$ is toric for some algebraic torus;
 \item $[v,w]$ contains no $2$-crown;
 \item $[v,w]$ is a lattice.
\end{enumerate}
\end{theorem}

The implication from \textup{(ii)} to \textup{(iii)} is due to Can--Saha \cite[Corollary~3.4]{CanSaha}. An intrinsic Richardson torus need not have a dense orbit under the standard diagonal action. Theorem~\ref{thm:braid-standard} identifies the standard-toric cases through smooth sliceness of their Richardson links, giving the following specialization.

\begin{corollary}\label{cor:finite-slice}
Let $v\leq w$ in $S_n$, let $\mathbf w$ be a reduced word for $w$, and let
$d=\ell(w)-\ell(v)$. The following are equivalent.
\begin{enumerate}[label=\textup{(\roman*)}]
 \item $\overline R_{v,w}$ is toric for the standard diagonal torus.
 \item $\Rfin_{v,w}$ is a single standard-torus orbit.
 \item $T_{v,\mathbf w}$ is a forest.
 \item $d=n-c(wv^{-1})$.
 \item The Richardson link
       $\Lambda_{v,w}=\widehat{\beta(w)\beta(v)^{-1}}$ is smoothly
       slice.
\end{enumerate}
Moreover $g_4(\Lambda_{v,w})=\bigl(d-n+c(wv^{-1})\bigr)/2$.
\end{corollary}

\begin{proof}[Proof of Corollary \ref{cor:finite-slice}]
Take $\beta=\mathbf w$ and $u=v$, so that $m=\ell(w)$ and the permutation represented by $\beta$ is $w$. Theorem \ref{thm:braid-standard} identifies standard toricity of $\Rfin_{v,w}$, the single-orbit property, the forest condition, the equality $d=n-c(wv^{-1})$, and smooth sliceness of $\Lambda_{v,w}$. Since $\Rfin_{v,w}$ is a dense stable open subset of $\overline R_{v,w}$, Lemma \ref{lem:stable-open} gives the equivalence with standard toricity of the closed Richardson variety. Equation \eqref{eq:braid-chi} specializes to the stated genus formula.
\end{proof}

For comparison with the Bruhat-interval-polytope formulation, choose a saturated chain
\[
 v=x_0\lessdot x_1\lessdot\cdots\lessdot x_d=w,
 \qquad x_j=x_{j-1}(a_j\ b_j),
\]
and form the multigraph $G_C$ on $[n]$ with one edge per cover. Let $\pi_0(G_C)$ be the set of connected components of $G_C$. The generic standard-torus orbit has dimension $n-\#\pi_0(G_C)$ \cite[Theorem~4.6]{tsukerman2015bruhat}; see also \cite[Section~4]{lee2021toric}.

\begin{corollary}[{\cite[Proposition~4.12]{tsukerman2015bruhat}}]\label{cor:finite-standard}
The conditions of Corollary \ref{cor:finite-slice} hold if and only if some, equivalently every, maximal-chain multigraph $G_C$ is a forest.
\end{corollary}

\begin{proof}
The closed Richardson variety has dimension $d$, so it is standard toric exactly when $n-\#\pi_0(G_C)=d$. As $G_C$ has $d$ edges counted with multiplicity, this is the forest condition, and the dimension formula applies to every saturated chain.
\end{proof}

\begin{remark}\label{rem:chain-solid}
The multigraphs $G_C$ and $T_{v,\mathbf w}$ need not be isomorphic. For $v=e$ and $\mathbf w=s_1s_2s_3$, the solid graph is the path with edges $12,23,34$, whereas the chain $e\lessdot s_3\lessdot s_2s_3\lessdot s_1s_2s_3$ gives the star with edges $34,24,14$; see Figure \ref{fig:chain-solid}.

For $\mathbf w=s_{i_1}\cdots s_{i_m}$, write $A_+=(p_1,\ldots,p_m)$ for the positive distinguished $v$-subword, and set $v_0=e$ and $v_j=p_1\cdots p_j$. Fill the omitted positions from right to left. At a position $j$ with $p_j=e$, write $q_j=s_{i_{j+1}}\cdots s_{i_m}$; the step is
\[
 x=v_{j-1}q_j\lessdot v_{j-1}s_{i_j}q_j=\tau_jx.
\]
PDS positivity and reducedness of the suffix $s_{i_j}q_j$ imply that each insertion is a cover. Thus the solid graph is the left-reflection graph of this chain, or equivalently the right-reflection graph of its inverse chain in $[v^{-1},w^{-1}]$. Tsukerman--Williams use right reflections \cite[Definition~4.5]{tsukerman2015bruhat}. Conjugating the subgroups generated by the left and right labels, and applying their chain-independence result \cite[Corollary~4.8]{tsukerman2015bruhat}, shows that $a\mapsto v(a)$ sends the components of $G_C$ to those of $T_{v,\mathbf w}$ for every $C$. Hence $c(G_C)=c(T_{v,\mathbf w})$, and the two forest criteria agree.
\end{remark}

\begin{figure}[ht!]
\centering
\begin{tikzpicture}[vertex/.style={circle,draw,fill=white,inner sep=0pt,minimum size=15pt,font=\small},line width=0.6pt]
\node at (1.65,1.95) {$T_{v,\mathbf w}$};
\node[vertex] (t1) at (0,0.6) {$1$};
\node[vertex] (t2) at (1.1,0.6) {$2$};
\node[vertex] (t3) at (2.2,0.6) {$3$};
\node[vertex] (t4) at (3.3,0.6) {$4$};
\draw (t1)--(t2)--(t3)--(t4);
\begin{scope}[xshift=6cm]
\node at (1.65,1.95) {$G_C$};
\node[vertex] (g1) at (0,0.05) {$1$};
\node[vertex] (g2) at (1.65,1.15) {$2$};
\node[vertex] (g3) at (3.3,0.05) {$3$};
\node[vertex] (g4) at (1.65,0.05) {$4$};
\draw (g4)--(g1) (g4)--(g2) (g4)--(g3);
\end{scope}
\end{tikzpicture}
\caption{The solid graph for $v=e$ and $\mathbf w=s_1s_2s_3$, and the graph of $C:e\lessdot s_3\lessdot s_2s_3\lessdot s_1s_2s_3$. Both are connected forests, but they are not isomorphic. See Remark~\ref{rem:chain-solid}}
\label{fig:chain-solid}
\end{figure}
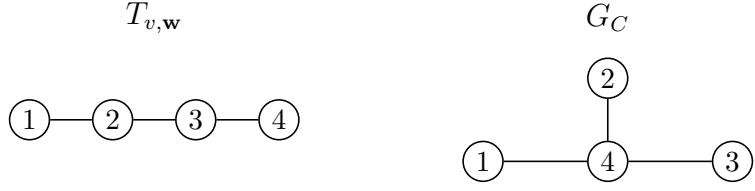

\subsection{The positroid specialization}

For positroid links, the slice criterion has a stronger consequence: smooth sliceness forces the link to be an unlink. We prove this by showing that the Seifert and slice genera agree.

Suppose that $w$ is $k$-Grassmannian, and let $f=f_{v,w}$ be the corresponding bounded affine permutation. The projection $\Rfin_{v,w}\to\Pio_f$ is a $T_n$-equivariant isomorphism \cite{knutson2013positroid,galashin2024positroids}, and $\Lambda_f:=\Lambda_{v,w}$ is the positroid link of $f$.

The grid-link presentation of \cite[Remark~2.2 and Theorem~2.5]{galashin2024plabic} gives an ordinary link diagram in $S^3$ with all $X$-markings on the northwest--southeast diagonal. Fixed points contribute split unknots and may be deleted before
drawing the grid \cite[Section~4.5]{galashin2024plabic}. An example of a grid link can be seen in Figure~\ref{fig:Grid-Link}.
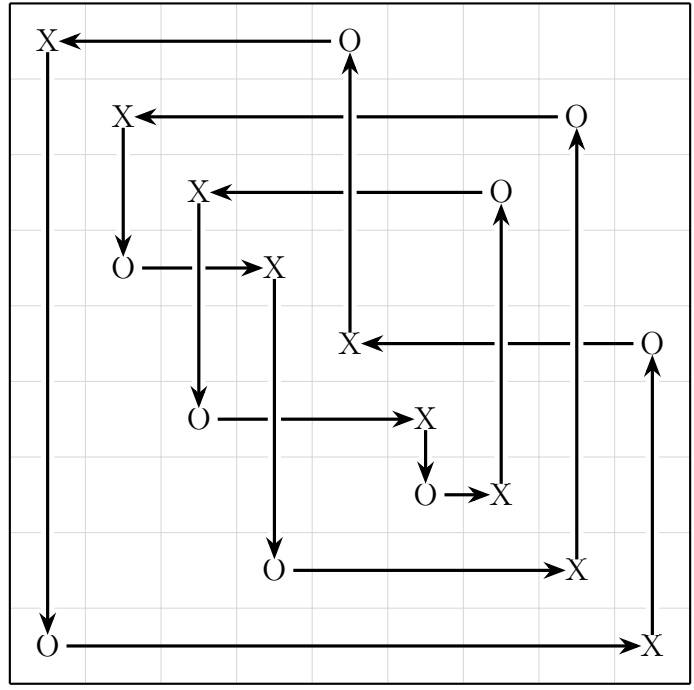
\begin{figure}
    \centering

\begin{tikzpicture}[
    scale=1.0,
    ->-/.style={postaction={decorate, decoration={markings, mark=at position 0.55 with {\arrow{>}}}}},
    ->>-/.style={postaction={decorate, decoration={markings, mark=at position 0.55 with {\arrow{>>}}}}},
    blue path/.style={blue!90!black, very thick, -{Stealth[length=2.5mm]}},
    red path/.style={red!90!black, very thick, -{Stealth[length=2.5mm]}},
    overcross/.style={preaction={draw, white, line width=5pt, -}}
]
\def\ht{9}
    \def\permA{9}
    \def\permB{6}
    \def\permC{4}
    \def\permD{7}
    \def\permE{1}
    \def\permF{8}
    \def\permG{3}
    \def\permH{2}
    \def\permI{5}

    \draw[gray!30, thin, step=1] (0,0) grid (\ht,\ht);

    \draw[thick] (0,0) -- (\ht,0);
    \draw[thick] (0,\ht) -- (\ht,\ht);
    \draw[thick] (0,0) -- (0,\ht);
    \draw[thick] (\ht,0) -- (\ht,\ht);

    \foreach \y in {1,...,\ht} { \node[label={[yshift=-12]X}] (X\y) at (\ht+0.5-\y,-0.5+\y) {}; }

    \node[label={[yshift=-12]O}] (O1) at (\ht+0.5- \permA,-0.5+ 1) {};
    \node[label={[yshift=-12]O}] (O2) at (\ht+0.5- \permB,-0.5+ 2) {};
    \node[label={[yshift=-12]O}] (O4) at (\ht+0.5- \permD,-0.5+ 4) {};

    \node[label={[yshift=-12]O}] (O3) at (\ht+0.5- \permC,-0.5+ 3) {};
    \node[label={[yshift=-12]O}] (O5) at (\ht+0.5- \permE,-0.5+ 5) {};
    \node[label={[yshift=-12]O}] (O6) at (\ht+0.5- \permF,-0.5+ 6) {};
    \node[label={[yshift=-12]O}] (O7) at (\ht+0.5- \permG,-0.5+ 7) {};
    \node[label={[yshift=-12]O}] (O8) at (\ht+0.5- \permH,-0.5+ 8) {};
    \node[label={[yshift=-12]O}] (O9) at (\ht+0.5- \permI,-0.5+ 9) {};

    \draw[-{Stealth}, very thick] (O1)+({-0.25*abs(1-\permA)/(1-\permA)},0) to (X1);
    \draw[-{Stealth}, very thick] (O2)+({-0.25*abs(2-\permB)/(2-\permB)},0) to (X2);
    \draw[-{Stealth}, very thick] (O4)+({-0.25*abs(4-\permD)/(4-\permD)},0) to (X4);
    \draw[-{Stealth}, very thick] (O3)+({-0.25*abs(3-\permC)/(3-\permC)},0) to (X3);
    \draw[-{Stealth}, very thick] (O5)+({-0.25*abs(5-\permE)/(5-\permE)},0) to (X5);
    \draw[-{Stealth}, very thick] (O6)+({-0.25*abs(6-\permF)/(6-\permF)},0) to (X6);
    \draw[-{Stealth}, very thick] (O7)+({-0.25*abs(7-\permG)/(7-\permG)},0) to (X7);
    \draw[-{Stealth}, very thick] (O8)+({-0.25*abs(8-\permH)/(8-\permH)},0) to (X8);
    \draw[-{Stealth}, very thick] (O9)+({-0.25*abs(9-\permI)/(9-\permI)},0) to (X9);

    \draw[-{Stealth}, very thick, overcross] (X\permD) to (O4);
    \draw[-{Stealth}, very thick, overcross] (X\permB) to (O2);
    \draw[-{Stealth}, very thick, overcross] (X\permA) to (O1);

    \draw[-{Stealth}, very thick, overcross] (X\permF) to (O6);
    \draw[-{Stealth}, very thick, overcross] (X\permE) to (O5);
    \draw[-{Stealth}, very thick, overcross] (X\permC) to (O3);
    \draw[-{Stealth}, very thick, overcross] (X\permG) to (O7);
    \draw[-{Stealth}, very thick, overcross] (X\permH) to (O8);
    \draw[-{Stealth}, very thick, overcross] (X\permI) to (O9);

\end{tikzpicture}

    \caption{Grid link of the permutation $(195)(268)(347)$.}
    \label{fig:Grid-Link}
\end{figure}

\begin{lemma}\label{lem:positroid-positive}
Every positroid link admits an oriented diagram with only positive crossings.
\end{lemma}

\begin{proof}
In the grid diagram above, vertical segments pass over horizontal segments, and the orientation runs vertically from $X$ to $O$ and horizontally from $O$ to $X$. Above the diagonal a crossing has its vertical segment oriented upward and its horizontal segment oriented leftward; below the diagonal both directions reverse. Both configurations are positive. This is the crossing argument of \cite[Theorem~3.3]{arndt2025diagonal}, which does not depend on the number of link components. Adding the split unknots for fixed points introduces no crossings.
\end{proof}

The following corollary combines Lemma \ref{lem:positroid-positive} with Rudolph's results on positive links and quasipositive Seifert surfaces \cite{rudolph1999positive,rudolph1998quasipositive}.

\begin{corollary}\label{cor:positroid-genera}
For every positroid link $\Lambda_f$, $g_3(\Lambda_f)=g_4(\Lambda_f)$.
\end{corollary}

\begin{proof}
By Lemma \ref{lem:positroid-positive}, the link $\Lambda_f$ admits a positive diagram. Rudolph's theorem \cite{rudolph1999positive} therefore gives a quasipositive Seifert surface $S\subset S^3$ with boundary $\Lambda_f$. By \cite[Lemma~5.2.1]{rudolph1998quasipositive}, its Euler characteristic satisfies $\chi_4(\Lambda_f)=\chi(S)$. Hence
\[
 \chi(S)\leq\chi_3(\Lambda_f)\leq\chi_4(\Lambda_f)=\chi(S).
\]
Thus $\chi_3(\Lambda_f)=\chi_4(\Lambda_f)$, and \eqref{eq:link-genera} gives the genus equality.
\end{proof}

The equality $g_3=g_4$ shows that a positroid link is smoothly slice exactly when it is an unlink. Subsection~\ref{subsec:slice-richardson} gives a nontrivial slice Richardson knot, showing that this conclusion does not extend to all Richardson links.

The orbit codimension also has a description in terms of a reduced plabic graph. Let $G$ be such a graph for $f$, with all $n$ boundary vertices included and with no disk-boundary arcs adjoined. Set $d=\dim\Pio_f$, let $r$ be the generic standard-torus orbit dimension, and let $c(G)$ be the number of connected components of $G$. Its cycle rank is $b_1(G)=|E(G)|-|V(G)|+c(G)$. The following formula gives the plabic forest criterion for standard toricity.
\begin{proposition}\label{prop:plabic-forest}
\begin{equation}\label{eq:plabic-complexity}
 d-r=b_1(G).
\end{equation}
\end{proposition}

\begin{proof}
Write $e=|E(G)|$, $v_{\mathrm{int}}=|V(G)|-n$, and $c=c(G)$. Every component meets the disk boundary. Adjoining its $n$ boundary arcs makes the graph connected, so Euler's formula gives $F=e-v_{\mathrm{int}}+1$ disk regions. The reduced plabic parametrization gives $d=F-1$, and graph components agree with positroid components \cite[Proposition~3.13 and the proof of Proposition~3.16]{lukowski2023positive}. The matroid-polytope orbit formula \cite{gelfand1987combinatorial} therefore gives $r=n-c$. Subtraction yields
\[
 d-r=e-v_{\mathrm{int}}-n+c=b_1(G).
\]
\end{proof}

The closed variety $\Pi_f=\overline{\Pio_f}$ is standard toric exactly when $d=r$, or equivalently when $G$ is a forest. This criterion is independent of the reduced representative. If $\cM(f)$ is connected, then $c(G)=1$ and the criterion is that $G$ is a tree.

\begin{proof}[Proof of Theorem \ref{thm:intro-positroid-link}]
Corollary~\ref{cor:positroid-genera} gives the genus equality and \textup{(iii)}$\Leftrightarrow$\textup{(iv)}. Under the equivariant Richardson--positroid isomorphism, Theorem \ref{thm:braid-standard} identifies the existence of a dense standard orbit on $\Pio_f$, the single-orbit property, and smooth sliceness of $\Lambda_f$. This proves \textup{(i)}$\Leftrightarrow$\textup{(iii)}.

Each nonempty patch $\Pio_{f,I}$ is a $T_n$-stable open subset of $\Pio_f$, because $\Delta_I$ is a semi-invariant. If $\Pio_f$ is a single orbit, every nonempty patch therefore equals $\Pio_f$. Conversely, if some patch is a single orbit, that orbit is dense in $\Pio_f$ by Lemma \ref{lem:stable-open}, so the preceding paragraph shows that $\Pio_f$ is a single orbit. This proves \textup{(i)}$\Leftrightarrow$\textup{(ii)}.
\end{proof}

\section{Affine Richardson varieties, patches, and Deodhar decompositions}
\label{sec:affine-patches}

We now return to intrinsic tori. On a fixed positroid variety, the standard action has a dense orbit on one nonempty Pl\"ucker patch exactly when it has a dense orbit on every nonempty patch. Whether a patch is an intrinsic torus can vary with the patch. Snider's isomorphism identifies each patch with an affine Richardson variety, allowing us to study this variation through the Deodhar decomposition. After recalling the affine setup, we prove the crown criterion in Sections~\ref{sec:crowns} and~\ref{sec:affine-proof}.

\subsection{Affine permutations and Richardson varieties}

An affine permutation is a bijection $f:\mathbb Z\to\mathbb Z$ with $f(i+n)=f(i)+n$. Its average is the integer $k$ determined by $\sum_{i=1}^n(f(i)-i)=kn$. Affine permutations of a fixed average form one component of the extended affine symmetric group. Right multiplication by $s_i$ interchanges $i+rn$ and $i+1+rn$ for every $r\in\mathbb Z$. In window notation the affine generator acts by
\[
 [a_1,\ldots,a_n]s_0=[a_n-n,a_2,\ldots,a_{n-1},a_1+n].
\]  If $f\leq g$ in Bruhat
order, left translation by their common length-zero factor identifies $[f,g]$ with an interval in the affine Coxeter group; we use this identification whenever reduced words or simple reflections are needed.

Let $X_g^\circ$ be the affine Schubert cell and $(X^f)^\circ$ the opposite Schubert cell. Their intersection
\[
  \Ro_{f,g}:=X_g^\circ\cap (X^f)^\circ
\]
is the open affine Richardson variety; it is smooth and irreducible of dimension $d=\ell(g)-\ell(f)$. Although the ambient affine flag variety is infinite dimensional, the Schubert variety $X_g$ and the closed Richardson variety $\overline{\cR}_f^g=X_g\cap X^f$ are finite-dimensional projective varieties.

\subsection{Bounded affine permutations and patches}

An affine permutation $f$ is $(k,n)$-bounded if it has average $k$ and satisfies $i\leq f(i)\leq i+n$ for every $i$. Write $\Bkn$ for the set of such permutations. For $I\in\binom{[n]}k$, the translation $t_I$ defined in Section~\ref{sec:intro} belongs to $\mathbf B_{k,n}$, and $\ell(t_I)=k(n-k)$.

The open positroid variety indexed by $f\in\mathbf B_{k,n}$ can be described from either the source or the target Grassmann necklace of $f$; we use only that it is a smooth irreducible variety of dimension $k(n-k)-\ell(f)$ and that each Pl\"ucker coordinate $\Delta_I$ is a $T_n$-semi-invariant. Consequently every nonempty $\Pio_{f,I}$ is a dense $T_n$-stable affine open subset of $\Pio_f$. See \cite{knutson2013positroid,galashin2023positroid,thomas}. We also use the standard fact that $\Pio_{f,I}\neq\varnothing$ if and only if $I$ is a basis of the positroid of $f$, if and only if $f\leq t_I$ \cite{snider,knutson2013positroid,thomas}.

Suppose $f\in\mathbf B_{k,n}$ and $f\leq t_I$. Snider \cite{snider} identifies the corresponding patch with an open affine Richardson variety.

\begin{theorem}\label{thm:snider}
There is an isomorphism
\[
  \Ro_{f,t_I}\xrightarrow{\ \sim\ }\Pio_{f,I}.
\]
\end{theorem}

\subsection{Deodhar decompositions}

Let $f\leq g$ be arbitrary affine permutations. Choose a reduced word $\mathbf g=s_{i_1}\cdots s_{i_m}$ for $g$ and let $\DS(f,\mathbf g)$ be the finite set of distinguished subexpressions $A=(a_1,\ldots,a_m)$ of $\mathbf g$ with product $f$, using the chosen-letter notation above. For $A\in\DS(f,\mathbf g)$ let $e(A)=\#\{j:a_j=e\}$, let $b(A)$ count the positions with $\ell(a_{(j)})<\ell(a_{(j-1)})$, and let $p(A)$ count the positions with $a_j=s_{i_j}$ and $\ell(a_{(j)})>\ell(a_{(j-1)})$. Then $m=p(A)+e(A)+b(A)$ and $\ell(f)=p(A)-b(A)$, so
\begin{equation}\label{eq:dimension-balance}
  e(A)+2b(A)=d.
\end{equation}

Deodhar's decomposition, extended to Kac--Moody groups by Billig--Dyer, takes the following form \cite{Deodhar,billig1994decompositions,MR04}.

\begin{theorem}\label{thm:affine-deodhar}
There is a decomposition into locally closed subvarieties
\begin{equation}\label{eq:deodhar-decomposition}
  \Ro_{f,g}
  =\bigsqcup_{A\in\DS(f,\mathbf g)}
  D_A,\qquad D_A\cong(\C^*)^{e(A)}\times\C^{b(A)}.
\end{equation}
\end{theorem}

Counting points over $\F_q$ gives
\begin{equation}\label{eq:R-deogram}
  R_{f,g}(q)
  =\sum_{A\in\DS(f,\mathbf g)}
    (q-1)^{e(A)}q^{b(A)}.
\end{equation}

There is a unique positive distinguished subexpression $A_+$ for $f$ in $\mathbf g$, with $b(A_+)=0$ and $e(A_+)=d$. Its stratum
\begin{equation}\label{eq:top-torus}
  D_+:=D_{A_+}\cong(\C^*)^d
\end{equation}
is the unique stratum of dimension $d$. Since $\Ro_{f,g}$ is irreducible, a locally closed stratum of full dimension is dense and therefore open. By \eqref{eq:dimension-balance} every other stratum has dimension $e(A)+b(A)=d-b(A)<d$, that is, codimension $b(A)$.

For the endpoint $g=t_I$ these subexpressions have a periodic diagrammatic model. Let $P_I$ be the up-right lattice path from $(0,0)$ to $(n-k,k)$ whose up-steps are indexed by $I$. A reading order on the boxes of a fundamental domain of the periodic strip determined by $P_I$ produces a reduced word for $t_I$ of length $k(n-k)$, and the $(f,P_I)$-affine Deograms are diagrammatic encodings of the distinguished subexpressions for $f$ in that word. We write $\Deo_{f,P_I}$ for their finite set; see \cite{thomas} for the definitions and for their independence of the chosen reading order. Under Snider's isomorphism \eqref{eq:deodhar-decomposition} becomes
\[
  \Pio_{f,I}
  =\bigsqcup_{A\in\Deo_{f,P_I}}
  (\C^*)^{e(A)}\times\C^{b(A)},
\]
and \eqref{eq:R-deogram} becomes the affine-Deogram formula for $R_{f,t_I}(q)$.

\section{Crown-free intervals and \texorpdfstring{$R$}{R}-polynomials}
\label{sec:crowns}

It remains to determine when an $R$-polynomial is a pure power of $q-1$. Brenti's crown criterion and Lemma~\ref{lem:codim-one} relate this condition to the Bruhat interval.

A \emph{$2$-crown} is a rank-three Bruhat interval with exactly two elements in each of its two intermediate ranks. Such an interval is isomorphic to $[e,s_1s_2s_1]$ in $S_3$. To see this, write the interval as $[a,b]$, with elements $c_1,c_2$ in the lower intermediate rank and $d_1,d_2$ in the upper one. Each interval $[a,d_i]$ has rank two and therefore exactly two intermediate elements. Hence both $c_1$ and $c_2$ lie below both $d_1$ and $d_2$.

An interval \emph{contains a $2$-crown} if one of its subintervals is isomorphic to this poset. In the terminology of \cite[Section~2.7]{bjorner2005combinatorics}, a $k$-crown is the face poset of a $k$-gon. Among crowns, only the $2$-crown fails to be a lattice.

Recall Deodhar's subword formula: for any Coxeter system, any $x\leq y$ and any reduced word $\mathbf y$ for $y$,
\begin{equation}\label{eq:deodhar-R}
  R_{x,y}(q)=\sum_{A\in\DS(x,\mathbf y)}(q-1)^{e(A)}q^{b(A)},
  \qquad e(A)+2b(A)=d:=\ell(y)-\ell(x),
\end{equation}
see \cite{Deodhar} and \cite[Chapter~5]{bjorner2005combinatorics}. The sum in \eqref{eq:deodhar-R} uses the same distinguished subexpressions as \eqref{eq:deodhar-decomposition}, and Lemma~\ref{lem:codim-one} applies without change.

\begin{proposition}\label{prop:crown-R}
Let $x\leq y$ in a Coxeter system, with $d=\ell(y)-\ell(x)\geq2$. The following are equivalent:
\begin{enumerate}[label=\textup{(\roman*)}]
\item $[x,y]$ contains no $2$-crown;
\item $[q^{d-1}]R_{x,y}(q)=-d$;
\item $R_{x,y}(q)=(q-1)^d$;
\item $\DS(x,\mathbf y)=\{A_+\}$ for some, equivalently every, reduced word $\mathbf y$ for $y$.
\end{enumerate}
\end{proposition}

\begin{proof}
Brenti's crown criterion \cite[Theorem~6.3]{brenti1994combinatorial} identifies \textup{(i)} with the coefficient condition $[q]R_{x,y}(q)=(-1)^{d-1}d$; see also \cite[Chapter~5, Exercise~36]{bjorner2005combinatorics}. The reciprocity identity
\[
 R_{x,y}(q)=(-q)^d R_{x,y}(q^{-1})
\]
converts this to \textup{(ii)}. The equivalence of \textup{(i)} and \textup{(iii)} is also part of the cited criterion; the following argument explains its relation to codimension-one strata. Group the terms of \eqref{eq:deodhar-R} by $b(A)$. Writing $m$ for the number of $A$ with $b(A)=1$, exactly as in Corollary \ref{cor:coefficient} we get
\[
 R_{x,y}(q)=(q-1)^d+m\,q(q-1)^{d-2}+\bigl(\text{degree}\leq d-2\bigr),
 \qquad [q^{d-1}]R_{x,y}(q)=-d+m .
\]
Thus \textup{(ii)} says $m=0$, which by Lemma \ref{lem:codim-one} forces $\DS(x,\mathbf y)=\{A_+\}$, that is \textup{(iv)}; and \textup{(iv)} gives \textup{(iii)} by \eqref{eq:deodhar-R}. Finally \textup{(iii)} gives \textup{(ii)} by inspection.
\end{proof}

\begin{remark}\label{rem:lattice}
A lattice interval cannot contain a $2$-crown, since every interval in a lattice is a lattice and a $2$-crown is not. The converse for arbitrary Bruhat intervals is attributed to an unpublished result of Dyer in \cite[Section~1]{unexpected}; that paper also proves the equivalence in finite type $A$. We retain the finite lattice equivalence in Theorem \ref{thm:GKSB}, but do not use the general converse or include it among the affine theorem's hypotheses and conclusions.
\end{remark}

\section{Intrinsic tori in affine type A}\label{sec:affine-proof}

We combine the affine Deodhar decomposition with Proposition~\ref{prop:crown-R} to prove Theorem~\ref{thm:main}. Snider's isomorphism then gives the corresponding criterion for positroid patches.

\begin{proof}[Proof of Theorem \ref{thm:main}]
Suppose first that $\Ro_{f,g}$ is an algebraic torus. Its open dense Deodhar stratum $D_+$ from \eqref{eq:top-torus} is then a torus of the same dimension, so Lemma \ref{lem:open-torus} gives $D_+=\Ro_{f,g}$. Hence the Deodhar decomposition has one stratum, the positive distinguished subexpression is the only distinguished subexpression, and $R_{f,g}(q)=(q-1)^d$ by \eqref{eq:R-deogram}.

Conversely, suppose $R_{f,g}(q)=(q-1)^d$ and substitute $q=1+z$ in \eqref{eq:R-deogram} to obtain
\begin{equation}\label{eq:positive-u}
  z^d=z^d+
  \sum_{A\neq A_+}z^{e(A)}(1+z)^{b(A)} .
\end{equation}
Every summand on the right is a nonzero polynomial with nonnegative integer coefficients, so there is no distinguished subexpression other than $A_+$ and $\Ro_{f,g}=D_+\cong(\C^*)^d$.

This proves the equivalence of \textup{(i)}, \textup{(iii)}, \textup{(iv)} and \textup{(v)}. Proposition \ref{prop:crown-R} gives the equivalence with \textup{(ii)} when $d\geq2$; for $d\leq1$ every interval is crown-free and $\Ro_{f,g}$ is a point or $\C^*$, so all five conditions hold. Independence of the chosen reduced word is automatic, since conditions \textup{(i)}--\textup{(iii)} do not refer to a word.
\end{proof}

\begin{proof}[Proof of Corollary \ref{cor:patch-main}]
Snider's isomorphism (Theorem \ref{thm:snider}) identifies $\Pio_{f,I}$ with $\Ro_{f,t_I}$. Apply Theorem \ref{thm:main} with $g=t_I$ and use the bijection between distinguished subexpressions of the $P_I$-word and $(f,P_I)$-affine Deograms.
\end{proof}

The same argument gives an affine version of Corollary~\ref{cor:coefficient}. Let $m_{f,g}$ count the Deodhar strata of codimension one in $\Ro_{f,g}$. For $d\geq2$, we obtain
\[
 R_{f,g}(q)
 =(q-1)^d+m_{f,g}\,q(q-1)^{d-2}
   +\bigl(\text{terms of degree at most $d-2$}\bigr),
\]
and $\Ro_{f,g}$ is a torus if and only if $m_{f,g}=0$, if and only if $[q^{d-1}]R_{f,g}(q)=-d$.

\section{Standard toricity, patches, and positroid links}
\label{sec:standard}

The affine criterion tells us which individual patches are intrinsic tori. We now combine it with the standard action. Since the existence of a dense standard orbit is the same on every nonempty patch of a fixed positroid variety, the two criteria distinguish standard torus patches from unexpected ones.

\begin{proposition}
\label{prop:standard-patch}
For a fixed $f\in\mathbf B_{k,n}$ the following are equivalent:
\begin{enumerate}[label=\textup{(\roman*)}]
  \item $\Pio_f$ has a dense $T_n$-orbit;
  \item some nonempty patch $\Pio_{f,I}$ has a dense $T_n$-orbit;
  \item every nonempty patch $\Pio_{f,I}$ has a dense $T_n$-orbit.
\end{enumerate}
\end{proposition}

\begin{proof}
Each nonempty patch $\Pio_{f,I}$ is a nonempty $T_n$-stable open subset
of the irreducible variety $\Pio_f$, since $\Delta_I$ is a
semi-invariant. Lemma \ref{lem:stable-open} applied to $X=\Pio_f$ and
$U=\Pio_{f,I}$ therefore gives
\textup{(i)}$\Leftrightarrow$\textup{(iii)}, while
\textup{(iii)}$\Rightarrow$\textup{(ii)} is immediate and the same lemma
gives \textup{(ii)}$\Rightarrow$\textup{(i)}.
\end{proof}

The generic $T_n$-orbit dimension has a matroidal description. If $M$ is
the matroid of a point of the Grassmannian, its matroid polytope has
dimension $n-\kappa(M)$, where $\kappa(M)$ is the number of connected
components, and this is also the dimension of the diagonal torus orbit
\cite{gelfand1987combinatorial}. Since the generic matroid on $\Pio_f$ is the positroid
$\cM(f)$,
\begin{equation}\label{eq:standard-rank}
  r_f:=\dim\bigl(T_n\cdot x_{\mathrm{gen}}\bigr)
  =n-\kappa\bigl(\cM(f)\bigr).
\end{equation}

\begin{corollary}\label{cor:standard-numerical}
The varieties in Proposition \ref{prop:standard-patch} are standard toric if
and only if
\[
  k(n-k)-\ell(f)=n-\kappa\bigl(\cM(f)\bigr).
\]
\end{corollary}

\begin{proof}
$\Pio_f$ is irreducible of dimension $k(n-k)-\ell(f)$, and a torus action
has a dense orbit exactly when the generic orbit has that dimension.
Substitute \eqref{eq:standard-rank} and apply Proposition
\ref{prop:standard-patch}.
\end{proof}

\begin{theorem}\label{thm:standard-patch-link}
Let $f\in\mathbf B_{k,n}$. The following are equivalent:
\begin{enumerate}[label=\textup{(\roman*)}]
 \item the closed positroid variety $\Pi_f$ is standard toric;
 \item $\Pio_f$ is a single $T_n$-orbit;
 \item some nonempty patch $\Pio_{f,I}$ is standard toric;
 \item every nonempty patch $\Pio_{f,I}$ is standard toric;
 \item some, equivalently every, reduced plabic graph for $f$ is a forest;
 \item the positroid link $\Lambda_f$ is smoothly slice;
 \item the positroid link $\Lambda_f$ is an unlink.
\end{enumerate}
When these conditions hold, every nonempty patch equals $\Pio_f$.
\end{theorem}

\begin{proof}
The equivariant Richardson--positroid isomorphism and Theorem \ref{thm:braid-standard} show that a dense standard orbit on $\Pio_f$ is the whole variety. Lemma \ref{lem:stable-open} therefore gives \textup{(i)}$\Leftrightarrow$\textup{(ii)}, and Proposition \ref{prop:standard-patch} identifies these with \textup{(iii)} and \textup{(iv)}. Proposition \ref{prop:plabic-forest} gives \textup{(i)}$\Leftrightarrow$\textup{(v)}, while Theorem \ref{thm:intro-positroid-link} gives the equivalence of \textup{(ii)}, \textup{(vi)}, and \textup{(vii)}.
\end{proof}

The forest test concerns standard toricity, not the existence of an intrinsic torus patch. More precisely, Proposition \ref{prop:plabic-forest} and Corollary \ref{cor:patch-main} show that a nonempty patch is an \emph{unexpected} torus exactly when $[f,t_I]$ is crown-free and $b_1(G)>0$. Thus cycles in the global plabic graph obstruct a dense standard orbit on every patch, but need not obstruct an intrinsic torus structure on a particular patch.

\begin{corollary}\label{cor:unexpected}
Let $f\leq t_I$ and put $d=k(n-k)-\ell(f)$ and $r_f=n-\kappa(\cM(f))$; then $r_f \leq d$.
\begin{enumerate}[label=\textup{(\alph*)}]
  \item The patch $\Pio_{f,I}$ is standard toric if and only if $d=r_f$;
        in that case it is automatically an algebraic torus, it equals
        $\Pio_f$, and $[f,t_I]$ is crown-free.
  \item The patch $\Pio_{f,I}$ is an unexpected algebraic torus if and
        only if $[f,t_I]$ is crown-free and $d>r_f$.
\end{enumerate}
\end{corollary}

\begin{proof}
The inequality $r_f\leq d$ is automatic, since $r_f$ is the dimension of an orbit in a $d$-dimensional variety. Part (a) is Corollary \ref{cor:standard-numerical} together with Theorem \ref{thm:standard-patch-link}, whose last assertion gives $\Pio_{f,I}=\Pio_f\cong(\C^*)^d$; crown-freeness then follows from Corollary \ref{cor:patch-main}. For (b), Corollary \ref{cor:patch-main} says that the patch is an algebraic torus exactly when $[f,t_I]$ is crown-free, and by (a) it fails to be standard toric exactly when $d\neq r_f$, that is $d>r_f$.
\end{proof}

\begin{remark}
If $\cM(f)$ is connected then $r_f=n-1$, and the condition $d>r_f$ becomes the dimension threshold used for unexpected Richardson varieties in \cite{unexpected}.
\end{remark}

\section{Examples}\label{sec:example}

We illustrate how intrinsic torus structure, standard torus orbits, and link topology interact. The first example shows that intrinsic toricity can vary between patches of a fixed positroid variety. The second exhibits an intrinsic Richardson torus with a nonslice link. We then compute the slice genus of the Perko knot from its generic orbit codimension and finish with a nontrivial slice Richardson knot, showing why the unlink characterization requires the positroid hypothesis.

\subsection{The smallest unexpected patch}\label{subsec:unexpected-patch}
We work in $\Gr(2,4)$ with Pl\"ucker coordinates $p_{12},p_{13},p_{14},p_{23},p_{24},p_{34}$ and the relation
\begin{equation}\label{eq:plucker-24}
  p_{12}p_{34}-p_{13}p_{24}+p_{14}p_{23}=0.
\end{equation}
Let
\[
  f=f_{2,4}=[3,4,5,6],\qquad f(i)=i+2,
\]
which has length $0$. The corresponding open positroid variety is the top-dimensional one,
\begin{equation}\label{eq:top-open-24}
  X:=\Pio_f
  =\{p_{12}p_{23}p_{34}p_{14}\neq0\}\subset\Gr(2,4),
\end{equation}
of dimension four, while the generic diagonal torus orbit has dimension three because the uniform matroid $U_{2,4}$ is connected. This is the first possible ambient size for an unexpected patch: for $k\in\{1,n-1\}$ every open positroid variety is a single standard-torus orbit, and for $n\leq3$ these are the only nontrivial ranks.

The usual reduced plabic graph is a square of four alternating black and white internal vertices, each joined to one boundary vertex, as pictured in Figure~\ref{fig:plabicHopf}. It has $8$ edges, $8$ vertices, and one component, so $b_1(G)=1=4-3$, in agreement with \eqref{eq:plabic-complexity}.

\begin{figure}
    \centering
    \begin{tikzpicture}[vertex/.style={circle,draw,fill=white,inner sep=0pt,minimum size=15pt,font=\small},overcross/.style={preaction={draw, white, line width=5pt, -}},link/.style={orange!90!black, very thick, -{Stealth[length=2.5mm]}},or/.style={orange!90!black, very thick},scale=0.8, line width=1pt]

    \filldraw[fill=white] (0,0) circle (2.5);

    \draw (1,1) -- (1,-1);
    \draw (1,1) -- (-1,1);
    \draw (-1,-1) -- (1,-1);
    \draw (-1,-1) -- (-1,1);
    
    \draw (-1,-1) -- (-1.75,-1.75);
    \draw (-1,1) -- (-1.75,1.75);
    \draw (1,-1) -- (1.75,-1.75);
    \draw (1,1) -- (1.75,1.75);
    
    \filldraw[fill=black] (1, 1.) circle (0.13);
    \filldraw[fill=white] (1, -1.) circle (0.13);
    \filldraw[fill=white] (-1, 1.) circle (0.13);
    \filldraw[fill=black] (-1, -1.) circle (0.13);

    \draw[link, overcross] (-0.5,1.5) to (0.5,0.5);
    \draw[link, overcross] (0.5,1.5) to (-0.5,0.5);

    \draw[link, overcross] (0.5,-1.5) to (-0.5,-0.5);
    \draw[link, overcross] (-0.5,-1.5) to (0.5,-0.5);
    
    \draw[link, overcross] (0.5,0.5) to (1.5,-0.5);
    \draw[link, overcross] (0.5,-0.5) to (1.5,0.5);

    \draw[link, overcross] (-0.5,0.5) to (-1.5,-0.5);
    \draw[link, overcross] (-0.5,-0.5) to (-1.5, 0.5);
    
    \draw[link, overcross] (1.5,0.5) [out=45,in=45] to (0.5,1.5);

    \draw[link, overcross] (-1.5,0.5) [out=135,in=135] to (-0.5,1.5);

    \draw[link, overcross] (-1.5,-0.5) [out=225,in=225] to (-0.5,-1.5);

    \draw[link, overcross] (1.5,-0.5) [out=315,in=315] to (0.5,-1.5);

    \end{tikzpicture}
    \caption{The plabic graph and plabic link for the top cell of $\Gr(2,4)$.}
    \label{fig:plabicHopf}
\end{figure}
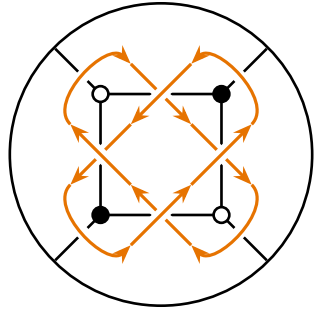

The positroid link of $f_{2,4}$ is the positive Hopf link \cite[Example~3.18]{galashin2024positroids}, the closure of $\sigma_1^2$, with two components, see Figure~\ref{fig:plabicHopf}. Thus
\[
 g_3(\Lambda_f)=g_4(\Lambda_f)
 =\frac{4-4+2}{2}=1,
\]
so by Theorem \ref{thm:standard-patch-link} neither $X$ nor any of its nonempty patches is standard toric. The computation below shows that this does not preclude an intrinsic, nonstandard torus structure on an individual patch.

\subsection*{The patch \texorpdfstring{$I=\{1,3\}$}{I=\{1,3\}}}

Here $t_{13}=[5,2,7,4]$, and factoring out the length-zero element $f$ gives the reduced word
\[
  t_{13}=f\,s_2s_0s_3s_1,\qquad \ell(t_{13})=4 .
\]
Every simple generator occurs at most once, so $[f,t_{13}]$ is the Boolean lattice $B_4$ and in particular is crown-free. Corollary \ref{cor:patch-main} therefore gives
\[
  \Pio_{f,13}\cong(\C^*)^4,
  \qquad R_{f,t_{13}}(q)=(q-1)^4 .
\]

This isomorphism is transparent in Pl\"ucker coordinates. On the chart $p_{13}\neq0$ the relation \eqref{eq:plucker-24} determines
\[
  \frac{p_{24}}{p_{13}}
  =\frac{p_{12}}{p_{13}}\cdot\frac{p_{34}}{p_{13}}
  +\frac{p_{14}}{p_{13}}\cdot\frac{p_{23}}{p_{13}},
\]
while the four cyclic ratios remain arbitrary and nonzero. Hence
\begin{equation}\label{eq:explicit-torus}
  \Pio_{f,13}\longrightarrow(\C^*)^4,
  \qquad
  x\longmapsto
  \left(\frac{p_{12}}{p_{13}},\frac{p_{23}}{p_{13}},
        \frac{p_{34}}{p_{13}},\frac{p_{14}}{p_{13}}\right)
\end{equation}
is an isomorphism: in the affine chart
\[
  \begin{pmatrix}1&a&0&b\\0&c&1&e\end{pmatrix},
  \qquad
  (p_{12},p_{13},p_{14},p_{23},p_{24},p_{34})
  =(c,\,1,\,e,\,a,\,ae-bc,\,-b),
\]
the four cyclic coordinates are $(c,a,-b,e)$, so their simultaneous nonvanishing cuts out $(\C^*)^4$ in $\C^4$. Since the standard orbit dimension is three, the four-dimensional torus structure in \eqref{eq:explicit-torus} is unexpected.

\subsection*{The patch \texorpdfstring{$I=\{1,2\}$}{I=\{1,2\}}}

By \eqref{eq:top-open-24} we have $\Pio_{f,12}=X$. Here $t_{12}=[5,6,3,4]$ and
\[
  t_{12}=f\,s_2s_3s_1s_2 .
\]
The subexpression $(s_2,s_3,e,s_2)$ has product $s_2s_3s_2$. Hence $[f,f\,s_2s_3s_2]$ is a $2$-crown subinterval of $[f,t_{12}]$. Equivalently,
\[
  R_{f,t_{12}}(q)
  =(q-1)^4+q(q-1)^2
  =(q-1)^2(q^2-q+1),
\]
so exactly one Deodhar stratum of codimension one occurs, in accordance with Lemma \ref{lem:codim-one}. Thus $\Pio_{f,12}$ is not an algebraic torus. This can be seen directly: in the chart $\begin{psmallmatrix}1&0&x&y\\0&1&z&w\end{psmallmatrix}$ one has $X=\{x\neq0,\ w\neq0,\ xw-yz\neq0\}\subset\C^4$, whose unit group modulo constants is free of rank three, generated by $x$, $w$ and $xw-yz$, rather than rank four. Over $\F_q$ this chart has $(q-1)^2\bigl(q^2-(q-1)\bigr)=(q-1)^2(q^2-q+1)$ points, matching $R_{f,t_{12}}$.

Rotating the example shows that the $p_{24}$-patch is also an unexpected four-torus, whereas all four cyclic Pl\"ucker patches coincide with $X$ and are not tori. Consequently no ``some patch if and only if every patch'' statement can hold for intrinsic torus structure, even though Proposition \ref{prop:standard-patch} gives exactly such a statement for standard toricity. It also shows why the global positroid link cannot detect the unexpected patch: the same nonslice Hopf link belongs to the ambient open positroid variety, while the crown condition changes with $I$.

\subsection{An unexpected Richardson torus}\label{subsec:unexpected-richardson}

The distinction between intrinsic toricity and sliceness also occurs for an entire open Richardson variety. In the example of Gorsky--Kim--Sherman-Bennett \cite[Example~2.11]{unexpected}, take
\[
 v=1324,\qquad w=4231\qquad\text{in }S_4,
\]
with reduced expression $w=s_1s_2s_3s_2s_1$. The unique distinguished subexpression for $v$ is $A_+=(e,e,e,s_2,e)$, so
\[
 \Rfin_{v,w}\cong(\C^*)^4,\qquad R_{v,w}(q)=(q-1)^4.
\]
This torus is unexpected, since standard diagonal orbits have dimension at most three.

The Richardson link
\[
 \Lambda_{v,w}=\widehat{\sigma_1\sigma_2\sigma_3\sigma_2\sigma_1\sigma_2^{-1}}
\]
is the mirror of the oriented Solomon link $L4a1\{0\}$ tabulated in LinkInfo \cite{linkinfo}; see \href{https://knotinfo.org/linkinfo/diagram_display.php?L4a1%7B0%7D}{its tabulated diagram}. Its two components have linking number $2$, so it is not smoothly slice. Since $\ell(w)-\ell(v)=4$ and $wv^{-1}=(1\,4)(2\,3)$, the slice-genus formula gives $g_4(\Lambda_{v,w})=(4-4+2)/2=1$.

\subsection{The Perko knot}\label{subsec:positive-genus}

We apply the genus--codimension formula to the Perko knot, recovering its known smooth slice genus from the standard torus action. We write positive words using the braid generators $\sigma_i$. Take
\[
 n=3,\qquad u=s_1,\qquad \beta=\sigma_2\sigma_1^2\sigma_2^2\sigma_1^3\sigma_2.
\]
The knot $L_{u,\beta}$ is the positive chirality of the Perko knot. Indeed, the mirror of the braid recorded by \href{https://katlas.org/wiki/10_161}{Knot Atlas for $10_{161}$} is
\[
\sigma_1^3\sigma_2\sigma_1^{-1}\sigma_2\sigma_1^2\sigma_2^2.
\]
Cyclically moving its first five letters to the end gives $\beta\sigma_1^{-1}$.

The PDS for $u$ selects position $8$, and the closure permutation is $\beta u^{-1}=(1\,2\,3)$. The solid graph is connected and has edge multiplicities
\[
 \{1,2\}:4,\qquad \{1,3\}:1,\qquad \{2,3\}:3.
\]
Thus $d=9-1=8$ and $c=\kappa=1$. Theorem~\ref{thm:genus-codimension} gives
\[
 C_{u,\beta}=8-3+1=6,\qquad g_4(L_{u,\beta})=\frac{C_{u,\beta}}2=3.
\]
This knot has a positive diagram but no positive braid presentation \cite[Example~4.2]{stoimenow2003positive}. Hence the inverse permutation factor is essential to this example.

\subsection{A nontrivial slice Richardson knot}\label{subsec:slice-richardson}

The Perko calculation exhibits positive slice genus and positive generic orbit codimension. We finish with the zero-genus case: a nontrivial slice knot whose Richardson variety is a single standard-torus orbit. This shows that the unlink conclusion of Theorem~\ref{thm:intro-positroid-link} does not extend to general Richardson links.

\begin{figure}
  \centering
  \begin{tikzpicture}[scale=1.5,line width=1pt]
    \draw (0,0)--(5,0);
    \draw (0,0.5)--(0.5,0.5)
          (1,0.5)--(4,0.5)
          (4.5,0.5)--(5,0.5);
    \draw (0,1)--(0.5,1)
          (1.5,1)--(4,1)
          (4.5,1)--(5,1);
    \draw (0,1.5)--(1,1.5)
          (2,1.5)--(4,1.5)
          (4.5,1.5)--(5,1.5);
    \draw (0,2)--(1.5,2)
          (2,2)--(3.5,2)
          (4.5,2)--(5,2);
    \draw (0,2.5)--(3.5,2.5)
          (4,2.5)--(5,2.5);
    \draw (0,3)--(5,3);
    \draw (0,3.5)--(5,3.5);

    \foreach \y/\strand in {
      0/1,0.5/2,1/3,1.5/4,2/5,2.5/6,3/7,3.5/8}
      \node[left] at (0,\y) {$\strand$};

    \foreach \x/\low/\high in {
      0.25/0/0.5,
      2.25/2/2.5,
      2.75/2.5/3,
      3.25/3/3.5,
      3.75/1/1.5,
      4.75/1/1.5,
      4.75/0/0.5} {
      \draw (\x,\low)--(\x,\high);
      \filldraw[fill=black] (\x,\low) circle (2pt);
      \filldraw[fill=white] (\x,\high) circle (2pt);
    }

    \foreach \leftx/\rightx/\low/\high in {
      0.5/1/0.5/1,
      1/1.5/1/1.5,
      1.5/2/1.5/2,
      3.5/4/2/2.5,
      4/4.5/1.5/2,
      4/4.5/0.5/1} {
      \draw[thick]
        (\leftx,\high) to[out=0,in=180] (\rightx,\low);
      \draw[white,line width=5pt]
        (\leftx,\low) to[out=0,in=180] (\rightx,\high);
      \draw[thick]
        (\leftx,\low) to[out=0,in=180] (\rightx,\high);
    }
  \end{tikzpicture}
  \caption{The 3D plabic graph for $v=14365278$ and $w=52734681$.}
  \label{fig:slice-not-unknot}
\end{figure}
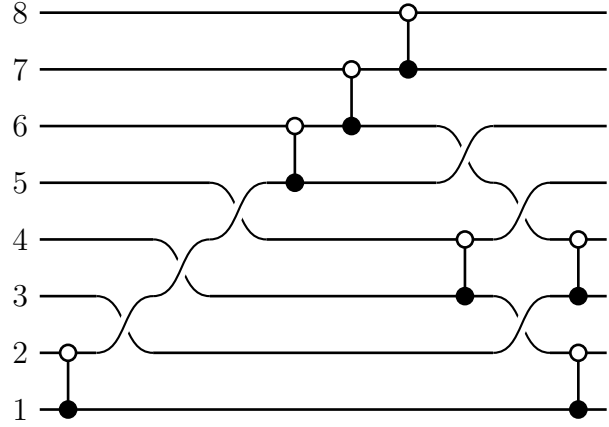

Take
\[
 v=14365278,\qquad w=52734681\qquad\text{in }S_8.
\]
Reduced expressions are
\[
 v=s_2s_3s_4s_5s_4s_2,\qquad w=s_1s_2s_3s_4s_5s_6s_7s_5s_3s_4s_2s_3s_1.
\]
The corresponding 3D plabic graph appears in Figure~\ref{fig:slice-not-unknot}. We have $v\leq w$, $\ell(v)=6$, $\ell(w)=13$, and
\[
 wv^{-1}=(1\,5\,4\,2\,6\,3\,7\,8),\qquad R_{v,w}(q)=(q-1)^7.
\]
Thus $\Lambda_{v,w}$ is a knot, and
\[
 g_4(\Lambda_{v,w})=\frac{7-8+1}{2}=0.
\]
By Corollary~\ref{cor:finite-slice}, the open Richardson variety is a single standard-torus orbit; its generic orbit codimension is zero, in agreement with Theorem~\ref{thm:genus-codimension}.

A sequence of Reidemeister moves identifies the closure of $\beta(w)\beta(v)^{-1}$ with the knot $12n_{582}$ in KnotInfo \cite{knotinfo}; see \href{https://knotinfo.org/diagram_display.php?12n_582}{its tabulated diagram}. Its Alexander polynomial is
\[
 \Delta_{\Lambda_{v,w}}(t)=t^{-2}-2t^{-1}+3-2t+t^2,
\]
so it is nontrivial. In particular, $g_4=0<g_3$.

\bibliographystyle{alpha}
\bibliography{surface_deep_pts}

\end{document}